\documentclass[oneside,english]{amsart}
\usepackage[T1]{fontenc}
\usepackage[utf8]{inputenc}
\usepackage{geometry}
\usepackage{amstext}
\usepackage{amsthm}
\usepackage{amssymb}
\usepackage{xypic}
\usepackage{graphicx}
\usepackage{xcolor}

\usepackage[normalem]{ulem}
\makeatletter

\providecommand{\tabularnewline}{\\}

\numberwithin{equation}{section}
\numberwithin{figure}{section}
\theoremstyle{plain}
  \theoremstyle{plain}
  \newtheorem*{thm*}{\protect\theoremname}
\theoremstyle{plain}
\newtheorem{thm}{\protect\theoremname}
  \theoremstyle{definition}
  \newtheorem{example}[thm]{\protect\examplename}
  \theoremstyle{definition}
  \newtheorem{defn}[thm]{\protect\definitionname}
  \theoremstyle{remark}
  \newtheorem{notation}[thm]{\protect\notationname}
  \theoremstyle{plain}
  \newtheorem{lem}[thm]{\protect\lemmaname}
  \theoremstyle{remark}
  \newtheorem{rem}[thm]{\protect\remarkname}
  \theoremstyle{plain}
  \newtheorem{prop}[thm]{\protect\propositionname}
  \theoremstyle{plain}
  
\theoremstyle{plain}
  \newtheorem{question}[thm]{\protect\questionname}

\makeatother

\usepackage{babel}
  \providecommand{\corollaryname}{Corollary}
  \providecommand{\definitionname}{Definition}
  \providecommand{\examplename}{Example}
  \providecommand{\lemmaname}{Lemma}
  \providecommand{\notationname}{Notation}
  \providecommand{\propositionname}{Proposition}
  \providecommand{\remarkname}{Remark}
\providecommand{\theoremname}{Theorem}
  \providecommand{\questionname}{Question}

\begin{document}
\subjclass[2010]{14R05 14R25 14E05 14P25 14J26}
\keywords{real algebraic model; affine line; rational fibration; birational diffeomorphism, Abhyankar-Moh } 

\author{Adrien Dubouloz}

\address{Adrien Dubouloz, IMB UMR5584, CNRS, Univ. Bourgogne Franche-Comté, F-21000 Dijon, France.}

\email{adrien.dubouloz@u-bourgogne.fr}

\author{Fr\'{e}d\'{e}ric Mangolte}

\address{Fr\'{e}d\'{e}ric Mangolte, Laboratoire angevin de recherche en math\'{e}matiques (LAREMA), CNRS, Universit\'{e} d\textquoteright{}Angers, F-49045 Angers, France}

\email{frederic.mangolte@univ-angers.fr}

\thanks{This research benefited at its early stage from the support of ANR Grant \textquotedbl{}BirPol\textquotedbl{}
ANR-11-JS01-004-01. The second author is partially supported by the ANR grant \textquotedbl{}ENUMGEOM\textquotedbl{} ANR-18-CE40-0009.}

\title{Algebraic models of the line in the real affine plane}
\begin{abstract}
We study smooth rational closed embeddings of the real affine line into the real affine plane, that is algebraic rational maps from the real affine line to the real affine plane which induce smooth closed embeddings of the real euclidean line into the real euclidean plane. We consider these up to equivalence under the group of birational automorphisms of the real affine plane which are diffeomorphisms of its real locus. We show that in contrast with the situation in the categories of smooth manifolds with smooth maps and of real algebraic varieties with regular maps where there is only one equivalence class up to isomorphism, there are non-equivalent smooth rational closed embeddings up to such birational diffeomorphisms. \end{abstract}

\maketitle

\section*{Introduction}
It is a consequence of the Jordan\textendash Schoenflies theorem (see Remark \ref{rem:JS-open} below) that every two smooth proper embeddings of $\mathbb{R}$ into $\mathbb{R}^{2}$ are equivalent up to composition by diffeomorphisms of $\mathbb{R}^2$.
Every algebraic closed embedding of the real affine line $\mathbb{A}_{\mathbb{R}}^{1}$ into
the real affine plane $\mathbb{A}_{\mathbb{R}}^{2}$ induces a smooth proper
embedding of the real locus $\mathbb{R}$ of $\mathbb{A}_{\mathbb{R}}^{1}$
into the real locus $\mathbb{R}^{2}$ of $\mathbb{A}_{\mathbb{R}}^{2}$.
Given two such algebraic embeddings $f,g:\mathbb{A}_{\mathbb{R}}^{1}\hookrightarrow\mathbb{A}_{\mathbb{R}}^{2}$, 
the famous Abhyankar-Moh Theorem \cite{AM75}, which is valid over any field of characteristic zero \cite[$\S$ 5.4]{vdE00}, asserts the existence
of a polynomial automorphism $\phi$ of $\mathbb{A}_{\mathbb{R}}^{2}$
such that $f=\phi\circ g$. This implies in particular that the smooth proper
embeddings of $\mathbb{R}$ into $\mathbb{R}^{2}$ induced
by $f$ and $g$ are equivalent under composition by a polynomial
diffeomorphism of $\mathbb{R}^{2}$. 

In this article, we consider a similar problem in a natural category intermediate between the real
algebraic and the smooth ones. Our main objects of study consist of
smooth embeddings of $\mathbb{R}$ into $\mathbb{R}^{2}$ induced
by rational algebraic maps $\mathbb{A}_{\mathbb{R}}^{1}\dashrightarrow\mathbb{A}_{\mathbb{R}}^{2}$
defined on the real locus of $\mathbb{A}_{\mathbb{R}}^{1}$ and whose
restrictions to this locus induce smooth proper embeddings of $\mathbb{R}$
into $\mathbb{R}^{2}$. We call these maps \emph{rational smooth embeddings},
and the question is the classification of these embeddings up to \emph{birational
diffeomorphisms} of $\mathbb{A}_{\mathbb{R}}^{2}$, that is, diffeomorphisms
of $\mathbb{R}^{2}$ which are induced by birational algebraic endomorphisms
of $\mathbb{A}_{\mathbb{R}}^{2}$ containing $\mathbb{R}^{2}$ in
their domains of definition and admitting rational inverses of the
same type. 

A first natural working question in this context is to decide whether
any rational smooth embedding is equivalent up to birational diffeomorphism
to the standard regular closed embedding of $\mathbb{A}_{\mathbb{R}}^{1}$
into $\mathbb{A}_{\mathbb{R}}^{2}$ as a linear subspace. Since every
rational smooth embedding $f:\mathbb{A}_{\mathbb{R}}^{1}\dashrightarrow\mathbb{A}_{\mathbb{R}}^{2}$
uniquely extends (see Lemma~\ref{lem:RatEmbed-from-projcurve}) 
to a morphism $\mathbb{P}_{\mathbb{R}}^{1}\rightarrow\mathbb{P}_{\mathbb{R}}^{2}$
birational onto its image, a rational smooth embedding which
can be rectified to a linear embedding by a birational diffeomorphism
defines in particular a rational plane curve $C$ that can be mapped onto a line
by a birational automorphism of $\mathbb{P}_{\mathbb{R}}^{2}$. By
classical results of Coolidge \cite{Co59}, Iitaka \cite{Ii83} and Kumar-Murthy
\cite{KuMu83}, complex curves with this property are characterized by
the negativity of the logarithmic Kodaira dimension of the complement
of their proper transform in a minimal resolution of their singularities.
Building on these ideas and techniques, we show the existence of \emph{non-rectifiable}
rational smooth embedding. In particular, we obtain the following
result, see Theorem \ref{thm:NonFib-Main} for a stronger statement: 
\begin{thm*}
For every integer $d\geq5$ there exists a non-rectifiable rational smooth embedding of $\mathbb{A}_{\mathbb{R}}^{1}$ into $\mathbb{A}_{\mathbb{R}}^{2}$ whose associated projective curve $C\subset\mathbb{P}_{\mathbb{R}}^{2}$ is a rational nodal curve of degree $d$.  
\end{thm*}
The existence of non-rectifiable rational smooth embeddings motivates
the search for weaker properties which can be satisfied by rational
smooth embeddings. To this end we observe that the Abhyankar-Moh Theorem
implies that the image of a regular closed embedding $\mathbb{A}_{\mathbb{R}}^{1}\hookrightarrow\mathbb{A}_{\mathbb{R}}^{2}$
is a real fiber of a structure of trivial $\mathbb{A}^{1}$-bundle
$\rho:\mathbb{A}_{\mathbb{R}}^{2}\rightarrow\mathbb{A}_{\mathbb{R}}^{1}$
on $\mathbb{A}_{\mathbb{R}}^{2}$. In the complex case, this naturally
leads to a ``generalized Abhyankar-Moh property'' for closed embeddings
of the affine line in affine surfaces $S$ equipped with $\mathbb{A}^{1}$-fibrations
over $\mathbb{A}_{\mathbb{C}}^{1}$, i.e. morphisms $\pi:S\rightarrow\mathbb{A}_{\mathbb{C}}^{1}$
whose general fibers are affine lines, which was studied for certain classes of surfaces in \cite{GMMR08}: the question there is whether the image of every regular closed embedding of $\mathbb{A}_{\mathbb{C}}^{1}$ in such
a surface is an irreducible component of a fiber of an $\mathbb{A}^{1}$-fibration.
The natural counterpart in our real birational setting consists in
shifting the focus to the question whether the image of a rational
smooth embedding is actually a fiber of an $\mathbb{A}^{1}$-fibration
$\pi:S\rightarrow\mathbb{A}_{\mathbb{R}}^{1}$ on a suitable real
affine surface $S$ birationally diffeomorphic to $\mathbb{A}_{\mathbb{R}}^{2}$,
but possibly non biregularly isomorphic to it. A rational smooth embedding
with this property is said to be \emph{biddable}. 

Being a fiber of an $\mathbb{A}^{1}$-fibration on a surface birationally
diffeomorphic to $\mathbb{A}_{\mathbb{R}}^{2}$ imposes strong restrictions
on the scheme-theoretic image $f_{*}(\mathbb{A}_{\mathbb{R}}^{1})$
of a rational smooth embedding $f:\mathbb{A}_{\mathbb{R}}^{1}\dashrightarrow\mathbb{A}_{\mathbb{R}}^{2}$.
We show in particular (Proposition \ref{prop:RealKodObst-fiber}) that
the \emph{real Kodaira dimension} \cite{BD17} $\kappa_{\mathbb{R}}(\mathbb{A}_{\mathbb{R}}^{2}\setminus f_{*}(\mathbb{A}_{\mathbb{R}}^{1}))$
of the complement of the image has to be negative, with the consequence
for instance that none of the rational smooth embedding mentioned
in the theorem above is actually biddable. 

In contrast, a systematic study of small degree embeddings (see in particular $\S$ \ref{subsec:quartic})
reveals that the images of these rational smooth embeddings are in a natural way smooth fibers of $\mathbb{A}^{1}$-fibrations
on some \emph{fake real planes}, a class of real birational models
of $\mathbb{A}_{\mathbb{R}}^{2}$ introduced and studied in the series of papers \cite{DuMa16,DuMa17}. These are smooth real
surfaces $S$ non isomorphic to $\mathbb{A}_{\mathbb{R}}^{2}$ whose
real loci are diffeomorphic to $\mathbb{R}^{2}$ and whose complexifications
have trivial reduced rational singular homology groups. We show moreover that smooth rational embeddings $\mathbb{A}_{\mathbb{R}}^{2}$ of degree less than or equal to four 
are in fact rectifiable, except possibly for two special quartic cases for which we are unable to conclude.

The scheme of the article is the following: section 1 contains preliminaries
on real algebraic varieties and a review of the structure of $\mathbb{A}^{1}$-\emph{fibered
algebraic models} of $\mathbb{R}^{2}$. In section 2, we introduce
the precise notions of rectifiable and biddable rational smooth embeddings
of the line in the plane and establish some basic properties. Section
3 is devoted to the construction of families of non-biddable rational
smooth embeddings.  Finally, section 4 contains a classification of rational
smooth embeddings of degree less than or equal to four.  \\

We want to thank Ilia Itenberg for his help in the construction of the nodal curves being used in Proposition \ref{prop:Curves-with-nonreal-nodes} and François Laudenbach for his help concerning equivalence classes of proper embeddings up to diffeomorphisms, see Remark \ref{rem:JS-open}. We are grateful to a referee for pointing out a mistake in an earlier version of the paper.

\section{Preliminaries}

A real (resp. complex) algebraic variety is a geometrically integral
scheme of finite type over the corresponding base field. A morphism
between such varieties is a morphism of schemes over the corresponding
base field. A complex algebraic variety $X$ is said to be defined
over $\mathbb{R}$ if there exists a real algebraic variety $X'$
and an isomorphism of complex algebraic varieties between $X$ and
the complexification $X'_{\mathbb{C}}=X'\times_{\mathbb{R}}\mathbb{C}$
of $X'$. 

\subsection{Euclidean topologies and birational diffeomorphisms}

The sets $X(\mathbb{R})$ and $X(\mathbb{C})$ of real and complex
points of a real algebraic variety $X$ are endowed in a natural way
with the Euclidean topology, locally induced by the usual Euclidean
topologies on $\mathbb{A}_{\mathbb{R}}^{n}\simeq\mathbb{R}^{2n}$
and $\mathbb{A}_{\mathbb{C}}^{n}\simeq\mathbb{C}^{n}$ respectively.
Every morphism $f\colon X\rightarrow Y$ of real algebraic varieties
induces a continuous map $f(\mathbb{R})\colon X(\mathbb{R})\rightarrow Y(\mathbb{R})$
for the Euclidean topologies, which is an homeomorphism when $f$
is an isomorphism. If $X$ and $Y$ are smooth, then $X(\mathbb{R})$
and $Y(\mathbb{R})$ can be further equipped with natural structures
of smooth $\mathcal{C}^{\infty}$-manifolds, and the continuous map $f(\mathbb{R})$
induced by an isomorphism of real algebraic varieties is a diffeomorphism
for these structures. See \cite[Sections 1.4 and 2.4]{Ma17} for details.
\begin{defn}
\label{def:RegularRregular} Let $\alpha:X\dashrightarrow Y$ be a
rational map between real algebraic varieties with non-empty real
loci $X(\mathbb{R})$ and $Y(\mathbb{R})$. 

a) We say that $\alpha$ is $\mathbb{R}$-\emph{regular} if $X(\mathbb{R})$
is contained in its domain of definition. 

b) We say that $\alpha$ is $\mathbb{R}$-\emph{biregular} if it is $\mathbb{R}$-regular and 
birational and its inverse is also $\mathbb{R}$-regular. 

\noindent For simplicity, an $\mathbb{R}$-biregular birational map between smooth
real algebraic varieties with non-empty real loci is called a \emph{birational
diffeomorphism}. 
\end{defn}

Note that when $X$ and $Y$ are smooth real varieties, a birational diffeomorphism between $X$ and $Y$ does restrict to a diffeomorphism between the real loci of $X$ and $Y$ equipped with their respective structure of smooth manifolds. 

\subsection{Pairs, log-resolutions and SNC completions }

A \emph{Simple Normal Crossing} $($SNC$)$ divisor on a smooth real
or complex surface $S$ is a curve $B$ whose complexification $B_{\mathbb{C}}$
has smooth irreducible components and ordinary double points only
as singularities. An SNC divisor $B$ on a smooth complete surface
$V$ is said to be \emph{SNC-minimal} if there does not exist
any projective strictly birational morphism $\tau:V\rightarrow V'$
onto a smooth complete surface defined over the same base field, with
exceptional locus contained in $B$ and such that the image $\tau(B)$
of $B$ is SNC. 

A \emph{smooth SNC pair} is a pair $(V,B)$ consisting of a smooth
projective surface $V$ and an SNC divisor $B$, both defined over
the considered base field.

A \emph{smooth completion} of a smooth algebraic surface $S$ is a
smooth SNC pair $(V,B)$ together with an isomorphism $V\setminus B\simeq S$,
all defined over the considered base field. Such completions always
exist as a consequence of classical completion and desingularization
theorems \cite{Na62,Wa35}.

A \emph{log-resolution} of a pair $(V,D)$ consisting of a surface
$V$ and a reduced curve $D$ on it is a projective birational morphism
$\tau:V'\rightarrow V$ defined over the considered base field, such
that $V'$ is smooth and the union of the reduced total transform
$\tau^{-1}(D)$ of $D$ with the exceptional locus $\mathrm{Ex}(\tau)$
of $\tau$ is an SNC divisor $B$ on $V'$. A log-resolution which is minimal
with respect to the ordering by birational domination is called
a \emph{minimal log-resolution} of $(V,D)$. 

By a minimal log-resolution of a birational map between smooth complete
surfaces, we mean a minimal log-resolution of its graph.

\subsection{$\mathbb{A}^{1}$-fibered algebraic models of $\mathbb{R}^{2}$}
In what follows, an \emph{algebraic model} of  $\mathbb{R}^{2}$ refers to a smooth geometrically
integral real surface $S$ with real locus $S(\mathbb{R})$ diffeomorphic to $\mathbb{R}^{2}$ and whose complexification
has the rational homology type of a point, i.e. has trivial reduced homology groups $\tilde{H}_{i}(S_{\mathbb{C}};\mathbb{Q})$ for every $i\geq0$. 
It turns out that $\mathbb{R}^{2}$ admits many algebraic models non-isomorphic to $\mathbb{A}_{\mathbb{R}}^{2}$, called \emph{fake real planes}, cf. \cite[Definition~1.1]{DuMa16}.
A partial classification of these was established in \cite{DuMa17} (see also \cite{DuMa16}) according to their usual Kodaira dimension. 
In the present article, algebraic models
$S$ of negative Kodaira dimension are of particular interest to us due to the fact that they admit an $\mathbb{A}^{1}$-\emph{fibration} $\pi:S\rightarrow\mathbb{A}_{\mathbb{R}}^{1}$
defined over $\mathbb{R}$, that is, a surjective morphism whose generic
fiber is isomorphic to the affine line over the function field of $\mathbb{A}^1_{\mathbb{R}}$. More
precisely, we have the following characterization:
\begin{thm}
\label{thm:-A1Fib-Model}$($\cite[Theorem 4.1]{DuMa17}$)$ A smooth
geometrically integral surface $S$ defined over $\mathbb{R}$ is
an algebraic model of $\mathbb{R}^{2}$ of Kodaira dimension $\kappa(S)=-\infty$
if and only if it admits an $\mathbb{A}^{1}$-fibration $\pi:S\rightarrow\mathbb{A}_{\mathbb{R}}^{1}$
defined over $\mathbb{R}$ whose only singular fibers, if any, are
geometrically irreducible real fibers of odd multiplicity, isomorphic
to $\mathbb{A}_{\mathbb{R}}^{1}$ when equipped with their reduced
structures. 
\end{thm}

\begin{rem} \label{prop:Top-fibers} The multiplicities of the singular fibers of an $\mathbb{A}^{1}$-fibered algebraic model of $\mathbb{R}^{2}$ are intimately related to the topology of its complexification (see e.g. \cite[Theorem 4.3.1 p. 231]{MiyBook}). Namely, letting $m_1 F_1, \ldots, m_rF_r$ be the singular fibers of $\pi$, where $F_i\simeq \mathbb{A}^1_{\mathbb{R}}$ and $m_i\geq 2$ for every $i=1,\ldots, r$, and $F$ be any fiber of $\pi$, we have  $$ H_1(S_{\mathbb{C}},\mathbb{Z})=\bigoplus_{F_i}\mathbb{Z}_{m_i} \quad \textrm{ and } \quad H_1((S\setminus F)_{\mathbb{C}},\mathbb{Z})=\mathbb{Z}\oplus \bigoplus_{F_i \neq F}\mathbb{Z}_{m_i}.$$
\end{rem}
Given an $\mathbb{A}^{1}$-fibration $\pi:S\rightarrow\mathbb{A}_{\mathbb{R}}^{1}$
on a smooth quasi-projective surface $S$, there exists a smooth projective
completion $(V,B)$ of $S$ on which $\pi$ extends to a morphism
$\overline{\pi}:V\rightarrow\mathbb{P}_{\mathbb{R}}^{1}$ with generic
fiber isomorphic to $\mathbb{P}^{1}$ over the function field of $\mathbb{P}_{\mathbb{R}}^{1}$,
called a $\mathbb{P}^{1}$-fibration. For $\mathbb{A}^{1}$-fibered
algebraic models of $\mathbb{R}^{2}$, we have the following more
precise description, see \cite[\S 4.1.1]{DuMa17}:

\begin{lem}
\label{lem:A1Fibmodel-completion}A smooth $\mathbb{A}^{1}$-fibered
algebraic model $\pi:S\rightarrow\mathbb{A}_{\mathbb{R}}^{1}$ of $\mathbb{R}^{2}$ 
admits a smooth projective completion $(V,B)$ on which $\pi$ extends
to a $\mathbb{P}^{1}$-fibration $\overline{\pi}:V\rightarrow\mathbb{P}_{\mathbb{R}}^{1}$
for which the divisor $B$ is a tree of rational curves that can be
written in the form 
\[
B=F_{\infty}\cup H\cup\bigcup_{p\in\mathbb{A}_{\mathbb{R}}^{1}(\mathbb{R})}G_{p}
\]
where:

a) $F_{\infty}\simeq\mathbb{P}_{\mathbb{R}}^{1}$ is the fiber of
$\overline{\pi}$ over $\infty=\mathbb{P}_{\mathbb{R}}^{1}\setminus\mathbb{A}_{\mathbb{R}}^{1}$, 

b) $H\simeq\mathbb{P}_{\mathbb{R}}^{1}$ is a section of $\overline{\pi}$, 

c) $G_{p}$ is either empty if $\pi^{-1}(p)$ is a smooth fiber of
$\pi$ or a proper SNC-minimal subtree of $\overline{\pi}^{-1}(p)$ containing
an irreducible component intersecting $H$ otherwise. 

Furthermore, for every $p$ such that $G_{p}\neq\emptyset$, the support
of $\overline{\pi}^{-1}(p)$ is a tree consisting of the union of $G_{p}$
with the support of the closure in $V$ of $\pi^{-1}(p)$. 
\end{lem}

Only little is known so far towards the classification of the above
$\mathbb{A}^{1}$-fibered algebraic models $\pi:S\rightarrow\mathbb{A}_{\mathbb{R}}^{1}$
of $\mathbb{R}^{2}$ up to birational diffeomorphism. At least, we
have the following criterion: 
\begin{thm}
\label{thm:rectif-criterion}$($\cite[Theorem 4.9 and Corollary 4.10]{DuMa17}$)$
Let $\pi:S\rightarrow\mathbb{A}_{\mathbb{R}}^{1}$ be an $\mathbb{A}^{1}$-fibered
algebraic model of $\mathbb{R}^{2}$. If all but at most one real
fibers of $\pi$ are reduced then $S$ is birationally diffeomorphic
to $\mathbb{A}_{\mathbb{R}}^{2}$. 
\end{thm}

\begin{example}
\label{exa:FakePlanes} Let $m\geq3$ be an odd integer and let $S_{m}$
be the smooth surface in $\mathbb{A}_{\mathbb{R}}^{3}$ defined by
the equation $u^{2}z=v^{m}-u$. Since $m$ is odd, the map $\mathbb{R}^{2}\rightarrow S_{m}(\mathbb{R})$,
$(u,z)\mapsto(u,\sqrt[m]{u^{2}z+u},z)$ is an homeomorphism, and since
$S_{m}$ is smooth, it follows that $S_{m}(\mathbb{R})$ is diffeomorphic
to $\mathbb{R}^{2}$. On the other hand, the projection $\pi=\mathrm{pr}_{u}:S_{m}\rightarrow\mathbb{A}_{\mathbb{R}}^{1}$
is an $\mathbb{A}^{1}$-fibration restricting to a trivial $\mathbb{A}^{1}$-bundle
over $\mathbb{A}_{\mathbb{R}}^{1}\setminus\{0\}$ and whose fiber
$\pi^{-1}(0)\simeq\mathrm{Spec}(\mathbb{R}[v]/(v^{m})[z])$ is isomorphic
to $\mathbb{A}_{\mathbb{R}}^{1}$, of multiplicity $m$. It is then easy to check that $H_{i}((S_{m})_{\mathbb{C}};\mathbb{Z})=0$ for every $i\geq2$, while
$H_{1}((S_{m})_{\mathbb{C}};\mathbb{Z})\simeq\mathbb{Z}_{m}$. Thus $S_{m}$ is an algebraic model of $\mathbb{R}^{2}$ non biregularly isomorphic to $\mathbb{A}_{\mathbb{R}}^{2}$ but which is birationally diffeomorphic to it by the previous theorem. 
\end{example}

\section{Rational closed embeddings of lines}

Given a rational map $f:X\dashrightarrow Y$ between real algebraic
varieties, we denote by $f_{*}(X)$ the scheme theoretic image of
the graph $\Gamma_{f}\subset X\times Y$ by the second projection. 
\begin{defn}
Let $X$ and $Y$ be real algebraic varieties with non-empty real
loci. 

1) An $\mathbb{R}$-regular rational map $f:X\dashrightarrow Y$ is
called an $\mathbb{R}$-\emph{regular closed embedding} if it is birational onto its image and $f\colon X(\mathbb{R})\rightarrow Y(\mathbb{R})$ is a closed embedding of topological manifolds, whose image coincides with $(f_{*}(X))(\mathbb{R})$. 

2) Two $\mathbb{R}$-regular closed embeddings $f:X\dashrightarrow Y$
and $g:X\dashrightarrow Y$ are called\emph{ equivalent} if there
exists   $\mathbb{R}$-biregular rational maps  $\alpha:X\dashrightarrow X$ and
$\beta:Y\dashrightarrow Y$ such that $\beta\circ f=g\circ\alpha$. 
\end{defn}

In the rest of the article, we focus on $\mathbb{R}$-regular closed
embeddings of the affine line $\mathbb{A}_{\mathbb{R}}^{1}=\mathrm{Spec}(\mathbb{R}[t])$
into the affine plane $\mathbb{A}^2_{\mathbb{R}}=\mathrm{Spec}(\mathbb{R}[x,y])$.

\begin{defn}
For simplicity, an $\mathbb{R}$-regular closed embedding of $\mathbb{A}_{\mathbb{R}}^{1}$ into a smooth real algebraic variety
is called a \emph{rational smooth embedding}. 
\end{defn}

We begin
with a series of examples illustrating this notion.

\subsection{First examples and non-examples }
\begin{example}
The rational map 
\[
f:\mathbb{A}_{\mathbb{R}}^{1}\dashrightarrow\mathbb{A}_{\mathbb{R}}^{2},\,t\mapsto(t,\frac{t^{3}}{t^{2}+1})
\]
is a rational smooth embedding. Indeed, since $f$ is the graph
of a rational function $\mathbb{A}_{\mathbb{R}}^{1}\dashrightarrow\mathbb{A}_{\mathbb{R}}^{1}$,
defined everywhere except at the complex point of $\mathbb{A}_{\mathbb{R}}^{1}$
with defining ideal $(t^{2}+1)\subset\mathbb{R}[t]$, it is a locally
closed embedding in the neighborhood of every real point of $\mathbb{A}_{\mathbb{R}}^{1}$.
The scheme theoretic image $f_{*}(\mathbb{A}_{\mathbb{R}}^{1})$ of
$f$ is equal to the closed curve $U=\left\{ x^{3}-(x^{2}+1)y=0\right\} \subset\mathbb{A}_{\mathbb{R}}^{2}$,
and $f$ induces an isomorphism $\mathbb{A}_{\mathbb{R}}^{1}(\mathbb{R})\cong\mathbb{R}\stackrel{\cong}{\rightarrow}U(\mathbb{R})\cong\mathbb{R}$. 
\end{example}

The next two examples enlighten the role of the condition $f\left(X(\mathbb{R})\right)=(f_{*}(X))(\mathbb{R})$
in the definition of an $\mathbb{R}$-regular closed embedding. 
\begin{example}
The rational map 
\[
f:\mathbb{A}_{\mathbb{R}}^{1}\dashrightarrow\mathbb{A}_{\mathbb{R}}^{2},\;t\mapsto(\frac{1-t^{2}}{t^{2}+1},\frac{2t}{t^{2}+1})
\]
is a locally closed embedding in the neighborhood of every real point
of $\mathbb{R}$ but not a rational smooth embedding.
Indeed, its scheme theoretic image is the conic $U=\left\{ x^{2}+y^{2}=1\right\} $
whose real locus is isomorphic to the circle $\mathbb{S}^{1}$, so
that the map $\mathbb{A}_{\mathbb{R}}^{1}(\mathbb{R})\cong\mathbb{R}\rightarrow f_{*}(\mathbb{A}_{\mathbb{R}}^{1})(\mathbb{R})\cong\mathbb{S}^{1}$
induced by $f$ is not an isomorphism. 
\end{example}

\begin{rem} \label{rem:JS-open} In the previous example, the smooth embedding $f(\mathbb{R}):\mathbb{R}\rightarrow \mathbb{R}^2$ induced by $f$ extends to an embedding  $S^1=\mathbb{R}\cup \{\infty\}\rightarrow \mathbb{R}^2$ thats maps the point $\infty$ onto the point $(-1,0)$, in particular $f(\mathbb{R})$ is not a proper embedding hence is not equivalent under composition by a diffeomorphism of $\mathbb{R}^2$ to a linear embedding.  In contrast, as mentioned in the introduction, it is a certainly well-known consequence of the Jordan-Schoenflies theorem that any two smooth proper embeddings of $\mathbb{R}$ into $\mathbb{R}^2$ are equivalent under composition by diffeomorphisms of $\mathbb{R}^2$. Let us give a short argument because of a lack of appropriate reference. First observe that every smooth proper embedding $f:\mathbb{R}\rightarrow \mathbb{R}^2$ uniquely extends to an injective continuous map $\overline{f}:S^1=\mathbb{R}\cup\{\infty\} \rightarrow S^2=\mathbb{R}^2\cup \{\infty \}$ such that $\overline{f}^{-1}(\infty)=\infty$. It follows that the closure in $S^2$ of the image $C$ of $f$ has a unique branch at the point $\infty$, which implies in turn that there exists a real number $R>0$ such that for every disc $D$ centered at the origin of $\mathbb{R}^2$ and of radius bigger than or equal to $R$, $C$  intersects the boundary of $D$ transversally in precisely two distinct points. Let $\{D_n\}_{n\in \mathbb{N}}$ be an exhaustion of $\mathbb{R}^2$ by a collection of such closed discs of strictly increasing radiuses. By the Jordan-Schoenflies theorem, there exists a diffeomorphism of $D_0$ which maps $C\cap D_0$ onto the horizontal diameter of $D_0$. We let $h_0$ be an extension of this diffeomorphism to a diffeomorphism of $\mathbb{R}^2$ which is the identity outside $D_1$. By construction, $f_0=h_0\circ f:\mathbb{R}\rightarrow \mathbb{R}^2$ is a smooth proper embedding and the intersection of its image with $D_0$ is equal to the horizontal diameter of $D_0$. Then by the Jordan-Schoenflies theorem again, we can find  a diffeomorphism of $D_2$ restricting to the identity on $D_0$ and mapping $h_0(C)\cap D_2$ onto the horizontal diameter of $D_2$. We then let $h_1$ be an extension of this diffeomorphism to a diffeomorphism of $\mathbb{R}^2$  restricting to the identity outside $D_3$ to obtain a smooth proper embedding $f_1=h_1\circ f_0:\mathbb{R}\rightarrow \mathbb{R}^2$  whose image intersects $D_2$ along the horizontal diameter of $D_2$. Iterating this procedure, we obtain a sequence of diffeomorphisms $h_n$, $n\geq 0$, of $\mathbb{R}^2$ restricting to the identity on $D_{2n-2}$ and outside $D_{2n+1}$ and a diffeomorphism $\varphi_n=h_n\circ \cdots \circ h_1 \circ h_0$ such that $\varphi_n\circ f:\mathbb{R}\rightarrow \mathbb{R}^2$ is a smooth proper embedding whose image intersects $D_{2n}$ along the horizontal diameter of $D_{2n}$. Since by construction the sequence of diffeomorphisms $\varphi_n=h_n\circ \cdots \circ h_1 \circ h_0$, $n\geq 0$, converges on every compact subset of $\mathbb{R}^2$, we conclude that there exists a diffeomorphism $\varphi$ of $\mathbb{R}^2$ such that the composition $\varphi\circ f : \mathbb{R}\rightarrow \mathbb{R}^2$ is the embedding of $\mathbb{R}$ as the horizontal coordinate axes. 
\end{rem} 

\begin{example}
The morphism $f:\mathbb{A}_{\mathbb{R}}^{1}=\mathrm{Spec}(\mathbb{R}[t])\rightarrow\mathbb{A}_{\mathbb{R}}^{2}=\mathrm{Spec}(\mathbb{R}[x,y]),\; t\mapsto(t^{2},t(t^{2}+1))$
induced by the normalization map of the nodal cubic $U\subset\mathbb{A}_{\mathbb{R}}^{2}$
defined by the equation $y^{2}=x\left(x+1\right)^{2}$ is a locally
closed immersion in a neighborhood of every real point of $\mathbb{A}_{\mathbb{R}}^{1}$ but not a rational smooth embedding. Indeed, $f(\mathbb{A}_{\mathbb{R}}^{1}(\mathbb{R}))$
is a proper subset of $f_{*}(\mathbb{A}_{\mathbb{R}}^{1})(\mathbb{R})=U\left(\mathbb{R}\right)=f(\mathbb{A}_{\mathbb{R}}^{1}(\mathbb{R}))\cup(-1,0)$.  

Note that the morphism $f(\mathbb{R}):\mathbb{R}\rightarrow \mathbb{R}^2$, $t\mapsto(t^{2},t(t^{2}+1))$ is a smooth proper embedding, which is thus equivalent in the category of smooth manifolds to a proper embedding as a linear subspace of $\mathbb{R}^2$. The following construction provides a diffeomorphism $\psi$ of $\mathbb{R}^2$ such that the composition $\psi\circ f(\mathbb{R})$ is a proper embedding as the graph of a smooth function, from which it is straightforward to deduce a smooth proper embedding as a linear subspace. Let $h:\mathbb{R}\rightarrow [0,1]$ be a smooth function such that $h(u)=1$ if $u\geq 0$ and $h(u)=0$ if $u\leq -\tfrac{1}{2}$. It then follows from the inverse function theorem that the map $$\varphi:\mathbb{R}^2\rightarrow \mathbb{R}^2, \quad (x,y)\mapsto (y,(1+yh(y))x)$$ is a diffeomorphism. Since by construction of $h$, we have $h(t^2)=1$ for every $t\in \mathbb{R}$, $\varphi$ maps the parametric curve $t\mapsto (t,t^2)$ onto the parametric curve $t\mapsto (t^2,(1+t^2h(t^2))t)=(t^2,t+t^3)$. So $\varphi^{-1}\circ f(\mathbb{R}):\mathbb{R}\rightarrow \mathbb{R}^2$ is the smooth proper embedding defined by $t\mapsto (t,t^2)$.
 \end{example}

\subsection{Rational embeddings of lines and projective rational plane curves}
\begin{notation}
Unless otherwise stated, in the rest of the article, we identify $\mathbb{A}_{\mathbb{R}}^{1}$
with the open complement in $\mathbb{P}_{\mathbb{R}}^{1}=\mathrm{Proj}(\mathbb{R}[u,v])$
of the $\mathbb{R}$-rational point $\left\{ v=0\right\} $, and $\mathbb{A}_{\mathbb{R}}^{2}$
with the open complement in $\mathbb{P}_{\mathbb{R}}^{2}=\mathrm{Proj}(\mathbb{R}\left[X,Y,Z\right])$
of the real line $L_{\infty}=\left\{ Z=0\right\} \simeq\mathbb{P}_{\mathbb{R}}^{1}$. 
\end{notation}

Every rational smooth embedding $f:\mathbb{A}_{\mathbb{R}}^{1}\dashrightarrow\mathbb{A}_{\mathbb{R}}^{2}$
extends to a morphism $\overline{f}:\mathbb{P}_{\mathbb{R}}^{1}\rightarrow\mathbb{P}_{\mathbb{R}}^{2}$
birational onto its image $C\subset\mathbb{P}_{\mathbb{R}}^{2}$,
such that $f_{*}(\mathbb{A}_{\mathbb{R}}^{1})=\overline{f}_{*}(\mathbb{A}_{\mathbb{R}}^{1})=C\cap\mathbb{A}_{\mathbb{R}}^{2}$
and $(f_{*}(\mathbb{A}_{\mathbb{R}}^{1}))(\mathbb{R})=(\overline{f}_{*}(\mathbb{A}_{\mathbb{R}}^{1}))(\mathbb{R})=(C\cap\mathbb{A}_{\mathbb{R}}^{2})(\mathbb{R})$.
Conversely, given a rational curve $C\subset\mathbb{P}_{\mathbb{R}}^{2}$,
the composition of the normalization morphism $\nu:\mathbb{P}_{\mathbb{R}}^{1}\rightarrow C$
with the inclusion $C\hookrightarrow\mathbb{P}_{\mathbb{R}}^{2}$
is a morphism $\overline{f}:\mathbb{P}_{\mathbb{R}}^{1}\rightarrow\mathbb{P}_{\mathbb{R}}^{2}$
birational onto its image $C$. The inverse image by $\nu$ of $C\cap L_{\infty}$
is a non-empty closed subset of $\mathbb{P}_{\mathbb{R}}^{1}$ defined over
$\mathbb{R}$. Its complement is a smooth affine rational curve $U$
on which $\overline{f}$ restricts to a proper morphism $f_{0}:U\rightarrow\mathbb{A}_{\mathbb{R}}^{2}$.
The following lemma characterizes rational planes curves $C$ whose associated morphisms $f_{0}:U\rightarrow\mathbb{A}_{\mathbb{R}}^{2}$ represent  rational smooth embeddings $f:\mathbb{A}_{\mathbb{R}}^{1}\dashrightarrow\mathbb{A}_{\mathbb{R}}^{2}$.

\begin{lem}
\label{lem:RatEmbed-from-projcurve} With the above notation, for a rational curve $C\subset\mathbb{P}_{\mathbb{R}}^{2}$,
the following are equivalent: 

a) $f_{0}:U\rightarrow\mathbb{A}_{\mathbb{R}}^{2}$ represents a rational
smooth embedding $f:\mathbb{A}_{\mathbb{R}}^{1}\dashrightarrow\mathbb{A}_{\mathbb{R}}^{2}$.

b) $C$ is smooth at every real point of $C\cap\mathbb{A}_{\mathbb{R}}^{2}$
and $U(\mathbb{R})=\mathbb{R}$. 

c) $C$ is smooth at every real point of $C\cap\mathbb{A}_{\mathbb{R}}^{2}$
and the inverse image of $C\cap L_{\infty}$ in the normalization of $C$ contains a unique real
point. 
\end{lem}
\begin{proof}

Since $U(\mathbb{R})=(\mathbb{P}_{\mathbb{R}}^{1}\setminus\overline{f}^{-1}(C\cap L_{\infty}))(\mathbb{R})$,
we see that $U(\mathbb{R})=\mathbb{R}$ if and only if $\overline{f}^{-1}(C\cap L_{\infty})$ contains a unique
real point. This proves the equivalence of b) and c). Now assume that
$f_{0}:U\rightarrow\mathbb{A}_{\mathbb{R}}^{2}$ represents a smooth
rational embedding $f:\mathbb{A}_{\mathbb{R}}^{1}\dashrightarrow\mathbb{A}_{\mathbb{R}}^{2}$.
Since $U$ is by definition the maximal affine open subset of $\mathbb{P}_{\mathbb{R}}^{1}$
on which $\overline{f}$ restricts to a morphism with image contained
in $\mathbb{A}_{\mathbb{R}}^{2}$, it follows that $U$ is the domain
of definition of $f:\mathbb{A}_{\mathbb{R}}^{1}\dashrightarrow\mathbb{A}_{\mathbb{R}}^{2}$.
This implies that $U(\mathbb{R})=\mathbb{A}_{\mathbb{R}}^{1}(\mathbb{R})=\mathbb{R}$.
Since $f_{*}(\mathbb{A}_{\mathbb{R}}^{1})=C\cap\mathbb{A}_{\mathbb{R}}^{2}$
and $f:\mathbb{A}_{\mathbb{R}}^{1}(\mathbb{R})=U(\mathbb{R})\rightarrow\mathbb{A}_{\mathbb{R}}^{2}(\mathbb{R})$
is a proper embedding of topological manifolds with image equal to
$(f_{*}(\mathbb{A}_{\mathbb{R}}^{1}))(\mathbb{R})=(C\cap\mathbb{A}_{\mathbb{R}}^{2})(\mathbb{R})$
we conclude that $C$ is smooth at every real point of $C\cap\mathbb{A}_{\mathbb{R}}^{2}$.

Conversely, assume that if $C$ is smooth at every real point of $C'=C\cap\mathbb{A}_{\mathbb{R}}^{2}$
and that $\overline{f}^{-1}(C\cap L_{\infty})$ contains a unique
real point, say $p_{\infty}\in\mathbb{P}_{\mathbb{R}}^{1}$. Then
$U$ is an open subset of $\mathbb{P}_{\mathbb{R}}^{1}\setminus\{p_{\infty}\}\simeq\mathbb{A}_{\mathbb{R}}^{1}$
with real locus $U(\mathbb{R})$ equal to $\mathbb{A}_{\mathbb{R}}^{1}(\mathbb{R})$
and the morphism $f_{0}:U\rightarrow\mathbb{A}_{\mathbb{R}}^{2}$,
which is the composition of the normalization $\nu:U\rightarrow C'$
with the closed immersion $C'\hookrightarrow\mathbb{A}_{\mathbb{R}}^{2}$,
represents an $\mathbb{R}$-regular rational map $f:\mathbb{A}_{\mathbb{R}}^{1}\dashrightarrow\mathbb{A}_{\mathbb{R}}^{2}$ birational onto its image and 
such that $f_{*}(\mathbb{A}_{\mathbb{R}}^{1})=C'$. Let $V\subset U$
be the inverse image by $\nu$ of the smooth locus $C'_{\mathrm{reg}}$
of $C'$ and let $\tilde{\nu}:V\rightarrow C'_{\mathrm{reg}}$ be
the induced isomorphism. Since $C'$ is smooth at every real point,
we have $C'_{\mathrm{reg}}(\mathbb{R})=C'(\mathbb{R})$. This implies
in turn that $V(\mathbb{R})=U(\mathbb{R})\simeq\mathbb{A}_{\mathbb{R}}^{1}(\mathbb{R})$
hence that $\tilde{\nu}:\mathbb{A}_{\mathbb{R}}^{1}(\mathbb{R})\simeq V(\mathbb{R})\rightarrow C'_{\mathrm{reg}}(\mathbb{R})$
is a diffeomorphism onto its image, which is equal to the closed submanifold
$C'(\mathbb{R})=(f_{*}(\mathbb{A}_{\mathbb{R}}^{1}))(\mathbb{R})$
of $\mathbb{A}_{\mathbb{R}}^{2}(\mathbb{R})$. This shows that $f:\mathbb{A}_{\mathbb{R}}^{1}\dashrightarrow\mathbb{A}_{\mathbb{R}}^{2}$
is a rational smooth embedding.
\end{proof}

\begin{example}
\label{ex:cubic} Let $\overline{f}:\mathbb{P}_{\mathbb{R}}^{1}\rightarrow\mathbb{P}_{\mathbb{R}}^{2}$,
$[u:v]\mapsto\left[uv^{2}:u^{3}:u^{2}v+v^{3}\right]$ be the morphism
induced by the normalization map of the irreducible cubic $C\subset\mathbb{P}_{\mathbb{R}}^{2}$
with equation $Z^{2}Y-X(X+Y)^{2}=0$. The intersection of $C$ with
$L_{\infty}$ consists of two real points: $p_{\infty}=[0:1:0]$ and
the unique singular point $p_{s}=\left[1:-1:0\right]$ of $C$, which
is a real ordinary double point with non-real complex conjugate tangents,
in other words, $\overline{f}^{-1}(p_{s})$ is a $\mathbb{C}$-rational
point of $\mathbb{P}_{\mathbb{R}}^{1}$. With the notation above,
we have $U\simeq\mathrm{Spec}(\mathbb{R}[t]_{t^{2}+1})$ and 
$f_{0}:U\rightarrow\mathbb{A}_{\mathbb{R}}^{2}$
is the morphism defined by 
\[
t\mapsto(\frac{t}{t^{2}+1},\frac{t^{3}}{t^{2}+1})
\]
Clearly $U(\mathbb{R})=\mathbb{A}_{\mathbb{R}}^{1}(\mathbb{R})$ and
so
$f_{0}:U\rightarrow\mathbb{A}_{\mathbb{R}}^{2}$
defines
an $\mathbb{R}$-regular map $f:\mathbb{A}_{\mathbb{R}}^{1}\dashrightarrow\mathbb{A}_{\mathbb{R}}^{2}$.
Since $p_{s}\in C\cap L_{\infty}$, the affine part $C\cap\mathbb{A}_{\mathbb{R}}^{2}\simeq\left\{ xy^{2}+2x^{2}y+x^{3}-y=0\right\} $
is a smooth curve, and so $f:\mathbb{A}_{\mathbb{R}}^{1}\dashrightarrow\mathbb{A}_{\mathbb{R}}^{2}$
is a rational smooth embedding. 
\end{example}

\subsection{Biddable and rectifiable lines}

\begin{defn}
\label{def:Rectif} A rational smooth embedding $f:\mathbb{A}_{\mathbb{R}}^{1}\dashrightarrow\mathbb{A}_{\mathbb{R}}^{2}$
is called: 

a) \emph{rectifiable} if it is equivalent to the linear closed
embedding $j_{\mathrm{lin}}:\mathbb{A}_{\mathbb{R}}^{1}\rightarrow\mathbb{A}_{\mathbb{R}}^{2}$,
$t\mapsto(t,0)$. 

b)\emph{ biddable} if there exists a birational
diffeomorphism $\alpha:\mathbb{A}_{\mathbb{R}}^{2}\dashrightarrow S$
onto an $\mathbb{A}^{1}$-fibered algebraic model $\pi:S\rightarrow\mathbb{A}_{\mathbb{R}}^{1}$
of $\mathbb{R}^{2}$ such that the induced rational smooth embedding $\alpha\circ f:\mathbb{A}_{\mathbb{R}}^{1}\dashrightarrow S$ 
is a regular closed embedding of $\mathbb{A}_{\mathbb{R}}^{1}$ as the support of a fiber of $\pi$. 
\end{defn}

\begin{rem}
\label{rem:StraightFiber}
Note that in case b) in Definition \ref{def:Rectif}, we do require that $\alpha\circ f$ is in fact a morphism $\mathbb{A}_{\mathbb{R}}^{1}\rightarrow S$, ond not only a rational map. 
A rectifiable rational smooth embedding is biddable since for any birational
diffeomorphism $\alpha$ of $\mathbb{A}_{\mathbb{R}}^{2}$ realizing
the equivalence with $j_{\mathrm{lin}}$, $\alpha\circ f$ coincides
with the closed immersion of $\mathbb{A}_{\mathbb{R}}^{1}$ as the
fiber of the trivial $\mathbb{A}^{1}$-bundle $\pi=\mathrm{pr}_{2}:\mathbb{A}_{\mathbb{R}}^{2}\rightarrow\mathbb{A}_{\mathbb{R}}^{1}$
over $0\in\mathbb{A}_{\mathbb{R}}^{1}(\mathbb{R})$.
\end{rem}

The following lemma is an immediate reformulation of Definition \ref{def:Rectif}: 
\begin{lem}
For an a rational smooth embedding $f:\mathbb{A}_{\mathbb{R}}^{1}\dashrightarrow\mathbb{A}_{\mathbb{R}}^{2}$
with associated irreducible rational curve $C=\overline{f}(\mathbb{P}_{\mathbb{R}}^{1})\subset\mathbb{P}_{\mathbb{R}}^{2}$,
the following hold:

a) $f$ is rectifiable if and only if there exists a birational
endomorphism of $\mathbb{P}_{\mathbb{R}}^{2}$ that restricts to a
birational diffeomorphism $\mathbb{P}_{\mathbb{R}}^{2}\setminus L_{\infty}\stackrel{\sim}{\dashrightarrow}\mathbb{P}_{\mathbb{R}}^{2}\setminus L_{\infty}$
and maps $C$ onto a line $\ell\simeq\mathbb{P}_{\mathbb{R}}^{1}$. 

b) $f$ is biddable if and only if there exists
a smooth pair $(V,B)$ with a $\mathbb{P}^{1}$-fibration $\overline{\pi}:V\rightarrow\mathbb{P}_{\mathbb{R}}^{1}$
as in Lemma \ref{lem:A1Fibmodel-completion} such that $\pi=\overline{\pi}|_{S}:S=V\setminus B\rightarrow\mathbb{A}_{\mathbb{R}}^{1}$
is an $\mathbb{A}^{1}$-fibered model of $\mathbb{R}^{2}$, and a
birational map $\mathbb{P}_{\mathbb{R}}^{2}\dashrightarrow V$ restricting
to a birational diffeomorphism $\mathbb{P}_{\mathbb{R}}^{2}\setminus L_{\infty}\stackrel{\sim}{\dashrightarrow}S$
that maps $C$ to an irreducible component of a fiber of $\overline{\pi}$. 
\end{lem}

\begin{example}
\label{exa:cubic-continued} Let $f:\mathbb{A}_{\mathbb{R}}^{1}\dashrightarrow\mathbb{A}_{\mathbb{R}}^{2}$
be the rational smooth embedding with associated curve
$C=\{Z^{2}Y-X(X+Y)^{2}=0\}\subset\mathbb{P}_{\mathbb{R}}^{2}$ constructed
in Example \ref{ex:cubic}. The birational endomorphism of $\mathbb{P}_{\mathbb{R}}^{2}$
defined by 
\[
[X:Y:Z]\mapsto[(X+Y)\left((X+Y)^{2}+Z^{2}\right):Z^{2}Y-X(X+Y)^{2}:Z\left((X+Y)^{2}+Z^{2}\right)]
\]
maps $L_{\infty}$ onto itself and restricts to a birational diffeomorphism of $\mathbb{P}^{2}\setminus L_{\infty}$. 
It contracts the union $\{(X+Y)^{2}+Z^{2}=0\}$ of the two non-real complex
conjugate tangents of $C$ at its singular point $p_{s}=[1:-1:0]$
onto the point $[0:1:0]\in L_{\infty}$, and maps $C$ onto the line
$\ell=\{Y=0\}$. It follows that $f:\mathbb{A}_{\mathbb{R}}^{1}\dashrightarrow\mathbb{A}_{\mathbb{R}}^{2}$
is rectifiable, a birational diffeomorphism $\alpha:\mathbb{A}_{\mathbb{R}}^{2}\dashrightarrow\mathbb{A}_{\mathbb{R}}^{2}$
such that $\alpha\circ f(t)=(t,0)$ being given for instance by 
\[
\left(x,y\right)\mapsto(x+y,-x+\frac{x+y}{(x+y)^{2}+1}).
\]
\end{example}

\subsection{\label{subsec:Real-Kodaira-dimension}Real Kodaira dimension: a numerical
obstruction to biddability}

\indent\newline 
Given an SNC divisor $B$ on a real smooth projective surface $V$, we denote by
$B_{\mathbb{R}}\subset B$ the union of all components of $B$ which
are defined over $\mathbb{R}$ and have infinite real loci. We say that $B$ has no \emph{imaginary loop} if
for every two distinct irreducible components $A$ and $A'$ of $B_{\mathbb{R}}$ the intersection
$A\cap A'$ is either empty or consists of real points only. 
If $B$ contains imaginary loops, then the reduced total transform $\hat{B}$ of $B$ in the blow-up $\tau:\hat{V}\rightarrow V$ of the set of non-real singular points of $B_\mathbb{R}$ is an SNC divisor with 
no
imaginary loops for which $\tau$ restricts to an isomorphism $\hat{V}\setminus \hat{B}\simeq V\setminus B$ (see \cite[Lemma 2.5]{BD17}). This implies in particular that every smooth quasi-projective real surface $S$ admits a smooth completion
$(V,B)$ such that that $B$ has no imaginary loop.

\begin{defn} \label{def:rKod} The \emph{real Kodaira dimension} $\kappa_{\mathbb{R}}(S)$ of a smooth quasi-projective real surface $S$ is the Iitaka dimension \cite{Ii70} $\kappa(V,K_{V}+B_{\mathbb{R}})$ of the pair $(V,B_{\mathbb{R}})$, where $(V,B)$ is any smooth completion
of $S$ such that that $B$ has no imaginary loop and $K_{V}$ is a canonical divisor on $V$
\end{defn}

It is established in \cite[Corollary 2.12]{BD17} that the so-defined element $\kappa_{\mathbb{R}}(S)\in\{-\infty,0,1,2\}$ is indeed independent on the
choice of a smooth completion $(V,B)$ of $S$ such that $B$ has no imaginary loop. Furthermore $\kappa_{\mathbb{R}}(S)$ is smaller than or equal
to the usual Kodaira dimension $\kappa(S)=\kappa(V,K_{V}+B)$, with equality in the case where $B=B_{\mathbb{R}}$, and is an invariant  of the isomorphism class of $S$ up to birational diffeomorphisms \cite[Proposition 2.13]{BD17}.

The following proposition then provides a simple numerical necessary
condition for biddability of a rational smooth embedding $f:\mathbb{A}_{\mathbb{R}}^{1}\dashrightarrow\mathbb{A}_{\mathbb{R}}^{2}$: 
\begin{prop}
\label{prop:RealKodObst-fiber} Let $f:\mathbb{A}_{\mathbb{R}}^{1}\dashrightarrow\mathbb{A}_{\mathbb{R}}^{2}$
be a rational smooth embedding and let $C$ be the image
of the corresponding morphism $\overline{f}:\mathbb{P}_{\mathbb{R}}^{1}\rightarrow\mathbb{P}_{\mathbb{R}}^{2}$.
If $f$ is biddable then $\kappa_{\mathbb{R}}(\mathbb{P}^{2}\setminus(C\cup L_{\infty}))=-\infty$. 
\end{prop}

\begin{proof}
If $f$ is biddable then by definition there exists a birational
diffeomorphism $\alpha:\mathbb{A}_{\mathbb{R}}^{2}\dashrightarrow S$
to a smooth $\mathbb{A}^{1}$-fibered model $\pi:S\rightarrow\mathbb{A}_{\mathbb{R}}^{1}$
of $\mathbb{R}^{2}$ such that $\alpha\circ f:\mathbb{A}_{\mathbb{R}}^{1}\rightarrow S$ is a closed embedding as the support 
of a fiber of $\pi$, say $\pi^{-1}(p)_{\mathrm{red}}$ for some point $p\in\mathbb{A}_{\mathbb{R}}^{1}(\mathbb{R})$. So
$\kappa_{\mathbb{R}}(\mathbb{P}_{\mathbb{R}}^{2}\setminus(C\cup L_{\infty}))=\kappa_{\mathbb{R}}(\mathbb{A}_{\mathbb{R}}^{2}\setminus C)=\kappa_{\mathbb{R}}(S\setminus \pi^{-1}(p))$ by invariance of the real Kodaira dimension under birational diffeomorphism.
Furthermore, since $\pi$ restricts to an $\mathbb{A}^{1}$-fibration on $S\setminus\pi^{-1}(p)$,
we have $\kappa_{\mathbb{R}}(\mathbb{A}_{\mathbb{R}}^{2}\setminus C)=\kappa_{\mathbb{R}}(S\setminus\pi^{-1}(p))\leq\kappa(S\setminus\pi^{-1}(p))=-\infty$. 
\end{proof}

\begin{question}
Is every rational smooth embedding $f:\mathbb{A}_{\mathbb{R}}^{1}\dashrightarrow\mathbb{A}_{\mathbb{R}}^{2}$
whose associated projective curve $C$ satisfies $\kappa_{\mathbb{R}}(\mathbb{P}_{\mathbb{R}}^{2}\setminus(C\cup L_{\infty}))=-\infty$
biddable ?
\end{question}

\subsection{\label{subsec:Bidd-rectif}Rectifiability of biddable lines: a sufficient criterion}

Recall that a rectifiable rational smooth embedding
$f:\mathbb{A}_{\mathbb{R}}^{1}\dashrightarrow\mathbb{A}_{\mathbb{R}}^{2}$
is in particular biddable. In this section, we establish a sufficient criterion for a biddable rational smooth embedding to be rectifiable. 

Let $f:\mathbb{A}_{\mathbb{R}}^{1}\dashrightarrow\mathbb{A}_{\mathbb{R}}^{2}$
be a biddable rational smooth embedding
with associated projective curve $C=\overline{f}(\mathbb{P}_{\mathbb{R}}^{1})\subset\mathbb{P}_{\mathbb{R}}^{2}$.
By Definition \ref{def:Rectif}, there exists a birational diffeomorphism
$\alpha:\mathbb{A}_{\mathbb{R}}^{2}\dashrightarrow S$ onto an $\mathbb{A}^{1}$-fibered
algebraic model $\pi:S\rightarrow\mathbb{A}_{\mathbb{R}}^{1}$ of
$\mathbb{R}^{2}$ such that the composition $\alpha\circ f:\mathbb{A}_{\mathbb{R}}^{1}\rightarrow S$
is a closed embedding of $\mathbb{A}_{\mathbb{R}}^{1}$ as the support of a fiber of $\pi$. 
Letting $(V,B)$ be a smooth
$\mathbb{P}^{1}$-fibered completion of $\pi:S\rightarrow\mathbb{A}_{\mathbb{R}}^{1}$
as in Lemma \ref{lem:A1Fibmodel-completion}, the composition $\overline{\pi}\circ\alpha:\mathbb{P}_{\mathbb{R}}^{2}\dashrightarrow\mathbb{P}_{\mathbb{R}}^{1}$ is a pencil of rational curves defined by a one-dimensional linear
system without fixed component $\mathcal{M}=(\overline{\pi}\circ\alpha)_{*}^{-1}|\mathcal{O}_{\mathbb{P}_{\mathbb{R}}^{1}}(1)|\subset|\mathcal{O}_{\mathbb{P}_{\mathbb{R}}^{2}}(d)|$ for some $d\geq1$, that has $C$ as an irreducible component of one
of its members. Every real member $M$ of $\mathcal{M}$ can be written as $M=M_{\mathbb{R}}+M'$ where $M_{\mathbb{R}}$ is an effective divisor whose support is equal to the union of all irreducible components of $M$ which are defined over $\mathbb{R}$ and have infinite real loci and where $M'$ is 
an effective divisor defined over $\mathbb{R}$ whose support does not contain any irreducible component of $M$ with infinite real locus.

\begin{lem} With the above notation, the following hold:

a) The linear system $\mathcal{M}$ has a real member $M_\infty$ such that the support of $M_{\infty,\mathbb{R}}$ is equal to $L_{\infty}$ 

b) For every real member $M\neq M_\infty$ of $\mathcal{M}$, the divisor $M_{\mathbb{R}}$ is an irreducible rational curve.
\end{lem}

\begin{proof}
Let $(V,B)$ be a smooth completion of $S$ on which $\pi$ extends
to a $\mathbb{P}^{1}$-fibration $\overline{\pi}:V\rightarrow\mathbb{P}_{\mathbb{R}}^{1}$
as in Lemma \ref{lem:A1Fibmodel-completion}. Since $\overline{\pi}$
restricts to an $\mathbb{A}^{1}$-fibration $\pi:S\rightarrow\mathbb{A}_{\mathbb{R}}^{1}$,
the real locus of the real member of $\mathcal{M}$ corresponding
to $\overline{\pi}^{-1}(\mathbb{P}_{\mathbb{R}}^{1}\setminus\mathbb{A}_{\mathbb{R}}^{1})$
is equal to $L_{\infty}$. This proves a). By Theorem \ref{thm:-A1Fib-Model}, every real fibers of $\pi$ contains a unique 
 geometrically irreducible component, which is isomorphic to
$\mathbb{A}_{\mathbb{R}}^{1}$ when equipped with its reduced structure.
Since every geometrically irreducible component of a real member of
$\mathcal{M}$ intersecting $\mathbb{A}_{\mathbb{R}}^{2}$ becomes
a geometrically irreducible component of a real fiber of $\pi$, assertion b) follows. 
\end{proof}

\begin{prop}
\label{thm:Main-rectif-criterion} Let $f:\mathbb{A}_{\mathbb{R}}^{1}\dashrightarrow\mathbb{A}_{\mathbb{R}}^{2}$
be a biddable rational smooth embedding with associated projective curve $C=\overline{f}(\mathbb{P}_{\mathbb{R}}^{1})\subset\mathbb{P}_{\mathbb{R}}^{2}$
and let $\mathcal{M}\subset|\mathcal{O}_{\mathbb{P}_{\mathbb{R}}^{2}}(d)|$ be the linear system associated to a birational diffeomorphism $\alpha:\mathbb{A}_{\mathbb{R}}^{2}\dashrightarrow S$ onto an $\mathbb{A}^{1}$-fibered algebraic model $\pi:S\rightarrow\mathbb{A}_{\mathbb{R}}^{1}$
of $\mathbb{R}^{2}$ such that $\alpha\circ f:\mathbb{A}_{\mathbb{R}}^{1}\rightarrow S$ is a closed immersion as the support of a fiber $F$ of $\pi$.
 Assume that one of the following two conditions is satisfied:

a) For every real member $M \neq M_\infty$ of $\mathcal{M}$, the curve $M_{\mathbb{R}}$ is irreducible and reduced. 

b) There exists a unique real member $M_0\neq M_\infty$ of $\mathcal{M}$ such that $M_{0,\mathbb{R}}$ is non-reduced and $C$ is the support of $M_{0,\mathbb{R}}$. 

Then $f:\mathbb{A}_{\mathbb{R}}^{1}\dashrightarrow\mathbb{A}_{\mathbb{R}}^{2}$ is rectifiable. 
\end{prop}

\begin{proof}
For every real member $M$ of $\mathcal{M}$, the curve $M_{\mathbb{R}}$ is a reduced component of $M$ if and only if $\alpha(M_{\mathbb{R}}\cap \mathbb{A}^2_{\mathbb{R}})$ is a reduced component of a real fiber of $\pi$. So by Lemma \ref{lem:A1Fibmodel-completion}, condition a) is equivalent to the property that all scheme theoretic fibers of $\pi:S\rightarrow\mathbb{A}_{\mathbb{R}}^{1}$ are isomorphic to $\mathbb{A}^{1}$ over the corresponding residue fields. This implies by \cite{KaMi78} that $\pi:S\rightarrow\mathbb{A}_{\mathbb{R}}^{1}$ is  isomorphic over $\mathbb{A}_{\mathbb{R}}^{1}$ to the trivial $\mathbb{A}^{1}$-bundle $\mathrm{pr}_{2}:\mathbb{A}_{\mathbb{R}}^{1}\times\mathbb{A}_{\mathbb{R}}^{1}\rightarrow\mathbb{A}_{\mathbb{R}}^{1}$. It follows that
$\alpha \circ f:\mathbb{A}_{\mathbb{R}}^{1}\rightarrow S\simeq \mathbb{A}^2_{\mathbb{R}}$ is  
a closed embedding of $\mathbb{A}_{\mathbb{R}}^{1}$ as a real fiber of $\mathrm{pr}_2$, hence that $f$ is rectifiable. 

Now suppose that condition b) if satisfied. Then $\pi:S\rightarrow\mathbb{A}_{\mathbb{R}}^{1}$ is an $\mathbb{A}^1$-fibration with a unique degenerate fiber, say $\pi^{-1}(0)$, and  $\alpha \circ f:\mathbb{A}_{\mathbb{R}}^{1}\rightarrow S$ is equal to the inclusion of $(\pi^{-1}(0))_{\mathrm{red}}\simeq \mathbb{A}^1_{\mathbb{R}}$ in $S$. By Theorem \ref{thm:rectif-criterion}, $S$ is birationally diffeomorphic to $\mathbb{A}_{\mathbb{R}}^{2}$. In fact, a careful tracing of the construction of a specific birational diffeomorphism $\beta:S\dashrightarrow\mathbb{A}_{\mathbb{R}}^{2}$ described
in \cite[Section 4.2]{DuMa17} shows that the composition of the inclusion $i_{0}:\mathbb{A}_{\mathbb{R}}^{1}\hookrightarrow S$
of $\pi^{-1}(0)_{\mathrm{red}}$ with $\beta$ is a linear regular closed embedding of $\mathbb{A}_{\mathbb{R}}^{1}$ into $\mathbb{A}_{\mathbb{R}}^{2}$. This implies that $\beta \circ \alpha \circ f:\mathbb{A}_{\mathbb{R}}^{1}\dashrightarrow \mathbb{A}_{\mathbb{R}}^{2}$ is
a linear regular closed embedding, hence that  $f$ is rectifiable.
\end{proof}

Let $S_{m}$ be a smooth affine surface with an $\mathbb{A}^{1}$-fibration
$\pi_{m}:S_{m}\rightarrow\mathbb{A}_{\mathbb{R}}^{1}$ defined over
$\mathbb{R}$ and admitting a unique degenerate real fiber, isomorphic
to $\mathbb{A}_{\mathbb{R}}^{1}$, of odd multiplicity $m\geq3$ (for
instance, take $S$ as in Example \ref{exa:FakePlanes}). By Theorem \ref{thm:rectif-criterion},
$S_{m}$ is an algebraic model of $\mathbb{R}^{2}$ birationally diffeomorphic
to $\mathbb{A}_{\mathbb{R}}^{2}$ but non-biregularly isomorphic to
it. Let $\alpha^{-1}:S_{m}\stackrel{\sim}{\dashrightarrow}\mathbb{A}_{\mathbb{R}}^{2}$
be such a birational diffeomorphism and let $i:\mathbb{A}_{\mathbb{R}}^{1}\hookrightarrow S_{m}$
be the inclusion of a general smooth real fiber of $\pi$. The composition $f_{m}=\alpha^{-1}\circ i:\mathbb{A}_{\mathbb{R}}^{1}\dashrightarrow\mathbb{A}_{\mathbb{R}}^{2}$ is then a biddable rational smooth embedding to which the criterion of Proposition \ref{thm:Main-rectif-criterion} does not apply. We do not know whether such rational smooth embedding $f_{m}:\mathbb{A}_{\mathbb{R}}^{1}\dashrightarrow\mathbb{A}_{\mathbb{R}}^{2}$ are rectifiable or not. 

\begin{question}
Does there exist biddable but non-rectifiable rational smooth embeddings $f:\mathbb{A}_{\mathbb{R}}^{1}\dashrightarrow\mathbb{A}_{\mathbb{R}}^{2}$? In particular, are the biddable rational smooth embedding $f_{m}:\mathbb{A}_{\mathbb{R}}^{1}\dashrightarrow\mathbb{A}_{\mathbb{R}}^{2}$ above rectifiable ?
\end{question}

\section{Families of non-biddable affine lines }

In this section, we exhibit examples of non-biddable,
hence in particular non-rectifiable, rational smooth embeddings $f:\mathbb{A}_{\mathbb{R}}^{1}\dashrightarrow\mathbb{A}_{\mathbb{R}}^{2}$
whose associated projective curves have any degree larger than or
equal to $5$. 
\begin{thm}
\label{thm:NonFib-Main}For every integer $d\geq5$ there exists a
non-biddable rational smooth embedding
$f\colon\mathbb{A}_{\mathbb{R}}^{1}\dashrightarrow\mathbb{A}_{\mathbb{R}}^{2}$
whose associated projective curve $C_{d}\subset\mathbb{P}_{\mathbb{R}}^{2}$
is a rational nodal curve of degree $d$. 
\end{thm}

The proof given below proceeds in two steps. We first construct in
Proposition \ref{prop:Curves-with-nonreal-nodes} real rational curves
$C_{d}\subset\mathbb{P}_{\mathbb{R}}^{2}$ having only ordinary nodes
as singularities for which the inclusion $C_{d}\cap\mathbb{A}_{\mathbb{R}}^{2}\hookrightarrow\mathbb{A}_{\mathbb{R}}^{2}$
defines a rational smooth embedding $f:\mathbb{A}_{\mathbb{R}}^{1}\dashrightarrow\mathbb{A}_{\mathbb{R}}^{2}$.
We then show by direct computation in Lemma \ref{lem:RealKod-comp}
that the real Kodaira dimension $\kappa_{\mathbb{R}}(\mathbb{P}^{2}\setminus (C_{d}\cup L_{\infty}))$
(see $\S$ \ref{subsec:Real-Kodaira-dimension}) is nonnegative, which
implies by virtue of Proposition \ref{prop:RealKodObst-fiber} that
$f:\mathbb{A}_{\mathbb{R}}^{1}\dashrightarrow\mathbb{A}_{\mathbb{R}}^{2}$
is not biddable.
\begin{prop}
\label{prop:Curves-with-nonreal-nodes} For any integer $d\geq1$
there exists a singular real rational curve $C_{d}\subset\mathbb{P}_{\mathbb{R}}^{2}$
of degree $d$ with only ordinary nodes as singularities with the
following properties:

1) If $d\equiv1$ or $2$ modulo $4$, $Sing(C_{d})$ consists of
pairs of non-real complex conjugate ordinary nodes. Furthermore, the intersection of $C_d$ with $L_\infty$ contains a unique real point and is transversal at every other point. 

2) If $d\equiv0$ or $3$ modulo $4$, $Sing(C_{d})$ consists of
pairs of non-real complex conjugate ordinary nodes and a unique real
node $p_{s}$ with non-real complex conjugate tangents.
Furthermore, the intersection of $C_d$ with $L_\infty$ contains $p_{s}$ and a unique other real
point, and is transversal at every other point. 
\end{prop}

\begin{proof}
The assertion is clear for $d=1,2$. We assume from now on that $d\geq 3$.  
 
1) If $d\equiv 1,2 \mod 4$, let $k=\frac12 (d-1)$ if $d$ is odd or $k=\frac12 (d-2)$ if $d$ is even. Observe that in any case, $k\ne0 $ is even, 
thus a general nodal rational complex curve $D_k\subset \mathbb{P}^2_{\mathbb{C}}$ of degree $k$ has no real point, meaning that $D_k\cap\mathbb{P}_{\mathbb{R}}^{2}(\mathbb{R})=\varnothing$. 
Let $D_k$ be such a curve given by an homogeneous complex polynomial $h$ of degree $k$.
A well-chosen projection $D_k$ to $\mathbb{P}^{2}_{\mathbb{C}}$ of the rational normal curve of degree $k$ in $\mathbb{P}^{k}_{\mathbb{C}}$ should suit. In particular, $D_k$ has $\frac 12(k-1)(k-2)$ non-real nodal points. The conjugated curve $\overline{D}_k$ of $D_k$ defined by the vanishing of the conjugated polynomial of $h$ 
intersects transversally $D_k$ in $k^2$ points 
Thus $D_k\cup \overline{D}_k$ is a reducible curve of degree $d-1$ defined over $\mathbb{R}$ with $(k-1)(k-2)+k^2$ non-real nodal points.
The union  $\tilde{C}_d=D_k\cup \overline{D}_k\cup E$, where $E$ is a general real line if $d$ is odd or a general real conic with non-empty real locus if $d$ is even, is a reducible curve of degree $d$ defined over $\mathbb{R}$. 

If $d$ is odd, $E$ meets $D_k\cup \overline{D}_k$ in $2k$ nodal points which are necessarily non-real, then $\tilde{C}_d$ has $k^2-\frac 12k+1$ pairs of non-real complex conjugate ordinary nodes. If $d$ is even, $E$ meets $D_k\cup \overline{D}_k$ in $4k$ nodal points then $\tilde{C}_d$ has $k^2+\frac 12k+1$ pairs of non-real complex conjugate ordinary nodes.
Given any pair of conjugate nodes of $(D_k\cap E)\cup (\overline{D}_k\cap E)$, it follows from Brusotti Theorem \cite[\S 5.5]{BR90} that there exists a small perturbation of $\tilde{C}_d$ which realizes a local smoothing of the later pair of nodes but preserves all the others. The resulting real curve $C_d$ is irreducible of degree $d$ with $\frac 14(d-1)(d-2)$ pairs of non-real conjugate nodes. 
The genus formula then implies that $C_d$ is geometrically rational. Furthermore, the real locus of $C_d$ contains smooth real points derived from the real locus of $E$ outside of $D_k\cup \overline{D}_k$. Thus $C_d$ is rational.
We now observe that the real locus of $C_d$ is a simple closed curve in the real projective plane which is one-sided if and only if $d$ is odd. If $d$ is odd (resp. $d$ is even) we deduce that there exists a real projective line $L_{\infty}$ meeting transversally $C_d(\mathbb{R})$ in one point $p$ and transversally at every other point (resp. meeting tangentially  $C_d(\mathbb{R})$ in one point $p$ and transversally at every other point).

  2) If $d\equiv 0,3 \mod 4$, let $k=\frac12 (d-1)$ if $d$ is odd or $k=\frac12 (d-2)$ if $d$ is even. Observe that in any case, $k$ is odd. 
As in  the preceding case, let $D_k\subset \mathbb{P}^2_{\mathbb{C}}$ be a general nodal rational curve given by an homogeneous polynomial $h$ of degree $k$.
As $k$ is odd, $D_k$ meets $\mathbb{P}_{\mathbb{R}}^{2}(\mathbb{R})$ in a unique point, say $p_s$. 
 
The same construction of a perturbation $C_d$ of $(D_k\cup \overline{D}_k\cup E)$ as in the preceding case works except that the line $L_\infty$ at the end of the construction is taken in the set of real lines passing through $p_s$.
\end{proof}

\begin{example}
Let us describe in more details the case $d=5$. Here $k=2$ and $D_k$ is a smooth conic. Take for example the complex conic given by 
$$
x^2+(1-\frac i3)y^2+(1-i)z^2=0
$$
Then $D_k\cap\mathbb{P}_{\mathbb{R}}^{2}(\mathbb{R})=\varnothing$ and the quartic $D_k\cup \overline{D}_k$ has exactly four non-real singular points $[\pm\sqrt 2,\pm i\sqrt 3,1]$ which are all nodes. Let $E$ be the real line of equation $y=0$. Then $\tilde{C}_d=D_k\cup \overline{D}_k\cup E$ is given by 
 $$
 2x^2y^3+2z^2x^2y+\frac83y^3z^2+yx^4+\frac{10}9y^5+2z^4y=0
 $$
 The curve $\tilde{C}_d$ has exactly eight singular points $[\pm\sqrt 2,\pm i\sqrt 3,1],[\pm\root4\of 2 e^{\pm i\frac{3\pi}8},0,1]$ which are all non-real nodes. Choosing a pair of conjugated nodes, say $[\root4\of 2 e^{\pm i\frac{3\pi}8},0,1]$ and perturbating them, we get a real quintic $C_d$ with one oval as real locus and three pairs of non-real complex conjugate ordinary nodes. Thus $C_d$ is a real rational curve.   
\end{example}

The following lemma completes the proof of Theorem \ref{thm:NonFib-Main}: 
\begin{lem}
\label{lem:RealKod-comp} For a real nodal rational curve $C_{d}\subset\mathbb{P}_{\mathbb{R}}^{2}$
of degree $d\geq5$ as in Proposition \ref{prop:Curves-with-nonreal-nodes},
we have $\kappa_{\mathbb{R}}(\mathbb{P}_{\mathbb{R}}^{2}\setminus(C_{d}\cup L_{\infty}))\geq0$. 
\end{lem}

\begin{proof}
It is well-known that if $d\geq6$, the complexification of $C=C_{d}$
cannot be mapped to a line by a birational automorphism of $\mathbb{P}_{\mathbb{C}}^{2}$. 
Indeed, since the proper transform of $C_{\mathbb{C}}$ in any resolution of
its singularities $\sigma:X\rightarrow\mathbb{P}_{\mathbb{C}}^{2}$
has self-intersection lower than or equal to $-2$, $C_{\mathbb{C}}$ cannot be contracted
to a point by a birational automorphism of $\mathbb{P}_{\mathbb{C}}^{2}$.

It then follows from a characterization attributed to Coolidge \cite{Co59}
(see also \cite{KuMu83}) that the Kodaira dimension $\kappa(X,K_{X}+\tilde{C}_{\mathbb{C}})$
is non-negative. Let $\tau:V\rightarrow\mathbb{P}_{\mathbb{R}}^{2}$
be a log resolution of the pair $(\mathbb{P}_{\mathbb{R}}^{2},C\cup L_{\infty})$
defined over $\mathbb{R}$ restricting to an isomorphism over $\mathbb{P}_{\mathbb{R}}^{2}\setminus(C\cup L_{\infty})$
and such that $B=\tau^{-1}(C\cup L_{\infty})_{\mathrm{red}}$ has
no imaginary loop. Since $\tilde{C}_{\mathbb{C}}$ is an irreducible
component of $(B_{\mathbb{R}})_{\mathbb{C}}$, we have 

\[
\kappa_{\mathbb{R}}(\mathbb{P}_{\mathbb{R}}^{2}\setminus(C\cup L_{\infty}))=\kappa(V_{\mathbb{C}},K_{V_{\mathbb{C}}}+(B_{\mathbb{R}})_{\mathbb{C}})\geq\kappa(V_{\mathbb{C}},K_{V_{\mathbb{C}}}+\tilde{C}_{\mathbb{C}})\geq0.
\]

In the remaining case $d=5$, the above argument does not longer work
since every rational nodal quintic can be mapped to a line by a birational
diffeomorphism of $\mathbb{P}_{\mathbb{R}}^{2}$ (see Remark \ref{rem:DiffBir-Quintic}
below). Nevertheless, a proof can be derived essentially along the
same lines as above from Iitaka's results \cite{Ii83,Ii88} on equivalence
classes of pairs of complex irreducible plane curves up to birational
automorphisms of $\mathbb{P}_{\mathbb{C}}^{2}$. But we find more
enlightening to provide a direct and self-contained argument. So here the curve
$C=C_{5}$ is a rational nodal quintic defined over $\mathbb{R}$
with only pairs of non-real complex conjugate ordinary nodes $\{q_{i},\overline{q}_{i}\}_{i=1,2,3}$
as singularities, which intersects $L_{\infty}$ transversally in
a real point $p$ and two pairs $\{p_{j},\overline{p}_{j}\}_{j=1,2}$
of non-real complex conjugate points. Let $\sigma_{1}:\mathbb{F}_{1}\rightarrow\mathbb{P}_{\mathbb{R}}^{2}$
be the blow-up of $p$ with exceptional divisor $C_{0}$. Denote by
$\ell_{i}$ and $\overline{\ell}_{i}$ the proper transforms of the
lines $[p,q_{i}]$ and $[p,\overline{q}_{i}]$ in $\mathbb{P}_{\mathbb{R}}^{2}$
and let $\mathbf{f}=\sum_{i=1}^{3}\ell_{i}+\overline{\ell}_{i}$.
Let $\sigma_{2}:V\rightarrow\mathbb{F}_{1}$ be the blow-up of the
pairs of points $\{q_{i},\overline{q}_{i}\}$ with exceptional divisors
$E_{i}$ and $\overline{E}_{i}$ and of the pairs $\{p_{j},\overline{p}_{j}\}$
with exceptional divisors $F_{j}$ and $\overline{F}_{j}$. We let
$\mathbf{E}=\sum_{i=1}^{3}E_{i}+\overline{E}_{i}$, $\mathbf{F}=\sum_{j=1}^{2}F_{j}+\overline{F}_{j}$,
and we let $\mathbf{f}'$ and $L'$ be the proper transform of $\mathbf{f}$
and $L$ respectively. The composition $\tau=\sigma_{1}\circ\sigma_{2}:V\rightarrow\mathbb{P}_{\mathbb{R}}^{2}$
is a log-resolution of the pair $(\mathbb{P}_{\mathbb{R}}^{2},C\cup L_{\infty})$
defined over $\mathbb{R}$, restricting to an isomorphism over $\mathbb{P}_{\mathbb{R}}^{2}\setminus(C\cup L_{\infty})$
and such that 
\[
B=\tau^{-1}(C\cup L_{\infty})_{\mathrm{red}}=C'+C_{0}'+L'+\mathbf{E}+\mathbf{F}=B_{\mathbb{R}}+\mathbf{E}+\mathbf{F}
\]
 is an SNC divisor without imaginary loop. 

Set $\alpha=\frac{1}{3}$, $\lambda=\frac{2}{3}$, $\beta=\mu=1$
so that $6\alpha+\beta=3$ and $6\lambda+\mu=5$. Letting $\ell$
be the proper transform on $\mathbb{F}_{1}$ of a general real line
through $p$ and $L$ be the proper transform of $L$, we have 
\begin{align*}
K_{\mathbb{F}_{1}} & \sim-2C_{0}-3\ell\sim_{\mathbb{Q}}-2C_{0}-\alpha\mathbf{f}-\beta L\\
C & \sim4C_{0}+5\ell\sim_{\mathbb{Q}}4C_{0}+\lambda\mathbf{f}+\mu L
\end{align*}
where we identified $C$ with its proper transform on $\mathbb{F}_{1}$.
Then on $V$, we obtain 
\begin{align*}
K_{V} & \sim\sigma_{2}^{*}(K_{\mathbb{F}_{1}})+\mathbf{E}+\mathbf{F}\sim_{\mathbb{Q}}-2C_{0}'-\alpha\mathbf{f}'-\beta L'+(1-\alpha)\mathbf{E}+(1-\beta)\mathbf{F}\\
C' & \sim\sigma_{2}^{*}(C)-2\mathbf{E}-\mathbf{F}\sim_{\mathbb{Q}}4C_{0}'+\lambda\mathbf{f}'+\mu L'+(\lambda-2)\mathbf{E}+(\mu-1)\mathbf{F}
\end{align*}
where $C_{0}'$ and $C'$ denote the proper transforms of $C_{0}$
and $C$ respectively. We thus obtain 
\begin{align*}
2K_{V}+B_{\mathbb{R}} & \sim_{\mathbb{Q}}C_{0}'+(\lambda-2\alpha)\mathbf{f}'+(\mu-2\beta+1)L'+(\lambda-2\alpha)\mathbf{E}+(\mu-2\beta+1)\mathbf{F}\\
 & \sim C_{0}'
\end{align*}
which implies that $\kappa_{\mathbb{R}}(\mathbb{P}_{\mathbb{R}}^{2}\setminus (C\cup L_\infty))\geq0$
as desired. 
\end{proof}
\begin{rem}
\label{rem:DiffBir-Quintic} 

As explained in the proof above, for
every $d\geq6$, the rational curve $C_{d}$ constructed in Proposition
\ref{prop:Curves-with-nonreal-nodes} cannot be transformed into a
line by a birational automorphism of $\mathbb{P}_{\mathbb{R}}^{2}$.
This directly implies the weaker fact that the corresponding rational smooth embedding $\mathbb{A}_{\mathbb{R}}^{1}\dashrightarrow\mathbb{A}_{\mathbb{R}}^{2}$
is not rectifiable. In contrast, there exists a birational diffeomorphism
of $\mathbb{P}_{\mathbb{R}}^{2}$ that maps the nodal rational quintic
$C_{5}$ to a line: it consists of the blow-up $\sigma:V\rightarrow\mathbb{P}_{\mathbb{R}}^{2}$
of the three pairs $\{q_{i},\overline{q}_{i}\}_{i=1,2,3}$ of non-real
complex conjugate nodes of $C_{5}$, followed by the contraction $\tau:V\rightarrow\mathbb{P}_{\mathbb{R}}^{2}$
of the proper transforms of the three pairs of non-real complex conjugate
nonsingular conics passing through five of the six points blown-up.
The proper transform of $C_{5}$ and $L_{\infty}$ by $\tau\circ\sigma^{-1}$
are respectively a line and a quintic with three pairs of non-real
complex conjugate nodes. 
\end{rem}

\section{Classification up to degree four }

In this section, we study the rectifiability and biddability of rational smooth embeddings
$f:\mathbb{A}_{\mathbb{R}}^{1}\dashrightarrow\mathbb{A}_{\mathbb{R}}^{2}$
whose associated projective curves $C\subset\mathbb{P}_{\mathbb{R}}^{2}$
have degrees less than or equal to four. 

\subsection{Degree lower than or equal to $3$ }
\begin{prop} \label{prop:rectif-up3}
Every rational smooth embedding $f:\mathbb{A}_{\mathbb{R}}^{1}\dashrightarrow\mathbb{A}_{\mathbb{R}}^{2}$
whose associated curve $C\subset\mathbb{P}_{\mathbb{R}}^{2}$ has
degree $\leq3$ is rectifiable. 
\end{prop}

\begin{proof}
If $\deg C=1$ then $C$ is a real line in $\mathbb{P}_{\mathbb{R}}^{2}$
and so $f$ is rectifiable. If $\deg C=2$, then $C$ is a smooth
conic, whose intersection with $L_{\infty}$ must consists of a unique
real point $p_{\infty}$ by Lemma \ref{lem:RatEmbed-from-projcurve}. It follows
that $f:\mathbb{A}_{\mathbb{R}}^{1}\dashrightarrow\mathbb{A}_{\mathbb{R}}^{2}$
is actually a closed emdedding of $\mathbb{R}$-schemes. So $f:\mathbb{A}_{\mathbb{R}}^{1}\hookrightarrow\mathbb{A}_{\mathbb{R}}^{2}$
is rectifiable by an $\mathbb{R}$-automorphism of $\mathbb{A}_{\mathbb{R}}^{2}$
by virtue of the Abhyankar-Moh Theorem over $\mathbb{R}$. 

We now assume that $C$ is a rational cubic in $\mathbb{P}_{\mathbb{R}}^{2}$. Then $C$ has a unique singular point which is either an ordinary double point or a cusp of multiplicity two. Since this point is unique, it is a real point $p_s$ of $C$, and it follows from Lemma \ref{lem:RatEmbed-from-projcurve} that $p_s\in C\cap L_{\infty}$. The intersection multiplicity of $C$
and $L_{\infty}$ at $p_{s}$ is thus at least $2$, and since $C \cdot L_{\infty}=3$,
it follows that $C\cap L_{\infty}$ contains at most another point
distinct from $p_{s}$. This yields the following alternative: 

a) $C\cap L_{\infty}=\{p_{s}\}$. Suppose that $p_s$ is an ordinary double point of $C$. Then $L_\infty$ is the tangent to one of the two analytic branches of $C$ at $p_s$. The tangent to the other branch is then necessarily real, so that the inverse image of $p_s$ in the normalization of $C$ consists of two distinct real points, which is impossible by Lemma \ref{lem:RatEmbed-from-projcurve}. Thus $p_{s}$ is an ordinary cusp, with tangent $L_{\infty}$. Then $f:\mathbb{A}_{\mathbb{R}}^{1}\dashrightarrow\mathbb{A}_{\mathbb{R}}^{2}$ is a closed embedding of $\mathbb{R}$-schemes, hence is rectifiable by the Abhyankar-Moh Theorem over $\mathbb{R}$.

b) $C\cap L_{\infty}$ consists of $p_{s}$ and a second real point
$p_{0}$ of $C$. Since the intersection multiplicity of $C$ and $L_\infty$ 
at $p_{0}$ is equal to one, it follows that $p_0$ is a smooth point of $C$ at which $C$ and $L_\infty$ intersect transversally. The inverse of $C\cap L_\infty$ in the normalization of $C$ thus contains the inverse image of $p_0$ as a real point. It then follows from Lemma \ref{lem:RatEmbed-from-projcurve} that $p_{s}$ is an ordinary double point with non-real complex conjugate tangents $T$ and $\overline{T}$. Example \ref{ex:cubic} provides an illustration of this situation in which the rational smooth embedding under consideration is rectifiable. Let us show that this holds in general. Let $\tau_{1}:\mathbb{F}_{1}\rightarrow\mathbb{P}_{\mathbb{R}}^{2}$
be the blow up of $p_{s}$ with exceptional divisor $E$ and let $\rho_{1}:\mathbb{F}_{1}\rightarrow\mathbb{P}_{\mathbb{R}}^{1}$
be the $\mathbb{P}^{1}$-bundle structure on $\mathbb{F}_{1}$ induced
by the projection $\mathbb{P}_{\mathbb{R}}^{2}\dashrightarrow\mathbb{P}_{\mathbb{R}}^{1}$
from the point $p_{s}$. The proper transform of $L_{\infty}$ is
a real fiber of $\rho_{1}$ while $T$ and $\overline{T}$ are a pair
of non-real complex conjugate fibers. The proper transform of $C$
is a smooth rational curve which is a section of $\rho_{1}$ intersecting
$E$ in the pair of non-real complex conjugate points $\{q,\overline{q}\}=(T\cup\overline{T})\cap E$.
Let $\tau_{2}:\mathbb{F}_{1}\dashrightarrow\mathbb{F}_{3}$ be the
elementary transformation defined over $\mathbb{R}$ obtained by blowing-up
$q$ and $\overline{q}$ and contracting the proper transforms of
$T$ and $\overline{T}$. The proper transform of $C$ is then a section
of $\rho_{3}:\mathbb{F}_{3}\rightarrow\mathbb{P}_{\mathbb{R}}^{1}$
which does not intersect the proper transform of $E$. The composition
$\tau_{2}\circ\tau_{1}^{-1}:\mathbb{P}_{\mathbb{R}}^{2}\dashrightarrow\mathbb{F}_{3}$
induces a birational diffeomorphism $\alpha:\mathbb{A}_{\mathbb{R}}^{2}=\mathbb{P}_{\mathbb{R}}^{2}\setminus L_{\infty}\dashrightarrow\mathbb{F}_{3}\setminus(E\cup L_{\infty})\simeq\mathbb{A}_{\mathbb{R}}^{2}$
with the property that $\alpha\circ f:\mathbb{A}_{\mathbb{R}}^{1}\dashrightarrow\mathbb{A}_{\mathbb{R}}^{2}$
is a closed embedding of $\mathbb{R}$-schemes which is thus rectifiable
by the Abhyankar-Moh Theorem over $\mathbb{R}$. This implies in turn
that $f$ is rectifiable. 
\end{proof}

\subsection{The quartic case }\label{subsec:quartic}

Now we consider the case where the projective curve $C$ associated
to $f\colon\mathbb{A}_{\mathbb{R}}^{1}\dashrightarrow\mathbb{A}_{\mathbb{R}}^{2}$
is a quartic. The types and configurations of singularities that can
appear on, or infinitely near to, an irreducible rational complex
plane quartic $C_{\mathbb{C}}$ are completely determined by the two
classical formulas
\[
g=\frac{(d-1)(d-2)}{2}-\sum_{p}\frac{1}{2}m_{p}(m_{p}-1)\qquad\textrm{and}\qquad\sum_{q\mapsto p}m_{q}\leq m_{p}.
\]
In the first one, $p$ varies over all singularities of $C_{\mathbb{C}}$ including
infinitely near ones and $m_{p}$ denotes the multiplicity at $p$. In the second one $q$ varies in the first infinitely near neighborhood
of $p$. The above formulas permit nine possible types of singularities
in the complex case. Some additional sub-types arise in the real case
according to the type of branches of $C$ at a given singular point, see 
\cite{Ma33,GuUtTaj66,Na84}, these are summarized in Table
\ref{QuarticSing}. 

\begin{table}[htb!]
\begin{tabular}{|c|c|c|}
\hline 
Classical Name & Modern Name & Multiplicity sequence\tabularnewline
\hline 
\hline 
Ordinary Node & $A_{1}$, $A_{1}^{*}$ & $[2]$\tabularnewline
\hline 
Ordinary Cusp & $A_{2}$ & $[2]$\tabularnewline
\hline 
Tacnode & $A_{3}$,$A_{3}^{*}$ & $[2,2]$\tabularnewline
\hline 
Double Cusp & $A_{4}$ & $[2,2]$\tabularnewline
\hline 
Oscnode & $A_{5}$, $A_{5}^{*}$ & $[2,2,2]$\tabularnewline
\hline 
Ramphoid Cusp & $A_{6}$ & $[2,2,2]$\tabularnewline
\hline 
Ordinary triple point & $D_{4}$,$D_{4}^{*}$ & $[3]$\tabularnewline
\hline 
Tacnode Cusp & $D_{5}$ & $[3]$\tabularnewline
\hline 
Multiplicity $3$ Cusp & $E_{6}$ & $[3]$\tabularnewline
\hline 
Pair of imaginary ordinary nodes & $2A_{1}^{i}$ & $[2]$ and $[2]$ \tabularnewline
\hline 
Pair of imaginary ordinary cusps & $2A_{2}^{i}$ & $[2]$ and $[2]$ \tabularnewline
\hline 
\end{tabular} 
\medskip
\caption{Possible singularity types of a real rational quartic.}
\label{QuarticSing}
\end{table}
The modern notation in Table \ref{QuarticSing} refers to Arnold's one \cite{Ar76},
with the following additional convention introduced by Gudkov \cite{Gu88}: 

- If there is no asterisk in the notation of a point then the point
is real and all the branches centered in it are real. 

- If there is one asterisk in the notation of a point then it is real
and precisely two branches centered in this point are non-real complex
conjugate.

- The upper index  refers to a pair of non-real conjugate points of
the given type.\\

The property for a real rational quartic $C$ to be the associated
projective curve of a rational smooth embedding $f:\mathbb{A}_{\mathbb{R}}^{1}\dashrightarrow\mathbb{A}_{\mathbb{R}}^{2}$
imposes additional restrictions on its configuration of singular
points:

\begin{lem} \label{lem:SingTypes}
The possible configurations of singularities of a real plane quartic
$C$ associated with a rational smooth embedding $f:\mathbb{A}_{\mathbb{R}}^{1}\dashrightarrow\mathbb{A}_{\mathbb{R}}^{2}$
are the following $($see Table \ref{QuarticSing-RatEmb}$)$:

1) Three singular points: 

\quad a) $A_{1}^*+2A_{1}^{i}$ where  the real singular point of type $A_1^*$ belongs $C\cap L_{\infty}$ and $L_\infty$ also intersects $C$ with multiplicity two at a smooth real point of $C$.

\quad b) $A_{2}+2A_{1}^{i}$ or $A_{2}+2A_{2}^{i}$. In each case the real singular point of type $A_2$ belongs to $C\cap L_{\infty}$, and the tangent line to $C$ at this point is different from $L_\infty$.

2) Two singular points: $A_{2}+A_{3}^{*}$ or $A_{4}+A_{1}^{*}$. In each case, both singular points are real points of $C\cap L_{\infty}$, and the tangent line to $C$ at each of them is different from $L_\infty$.

3) A unique singular point: $A_{5}^{*}$, $A_{6}$, $D_{4}^{*}$ or $E_{6}$. In each case, the unique singular point is a real point of $C\cap L_{\infty}$. In the cases $D_{4}^{*}$ and $E_6$, $L_\infty$ is a tangent line to $C$ at the singular point whereas in the case  $A_{5}^{*}$, the tangent line to $C$ at the singular point is different from $L_{\infty}$. 
\end{lem}

\begin{table}[htb!]

\begin{tabular}{ccc}
\includegraphics[scale=1]{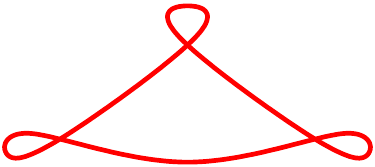} & \includegraphics[scale=1]{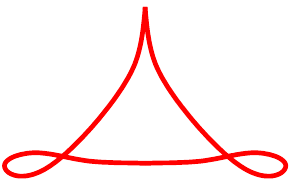} &  \includegraphics[scale=1]{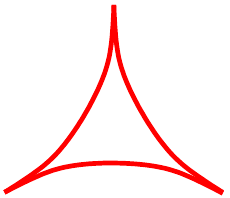}  \tabularnewline
$A_1^*+2A_1^i$ &  $A_2+2A_1^i$ & $A_2+2A_2^i$  \tabularnewline \tabularnewline
\multicolumn{3}{c}{Three singular points}\tabularnewline
\end{tabular}

\medskip

\begin{tabular}{cc}
\includegraphics[scale=1]{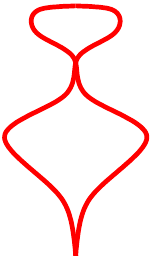} &  \includegraphics[scale=1]{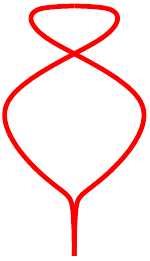} \tabularnewline
$A_2+A_3^*$ & $A_4+A_1^*$ \tabularnewline\tabularnewline
\multicolumn{2}{c}{Two singular points}\tabularnewline
\end{tabular}

\medskip

\begin{tabular}{cccc}
\includegraphics[scale=1]{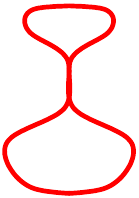} & \includegraphics[scale=1]{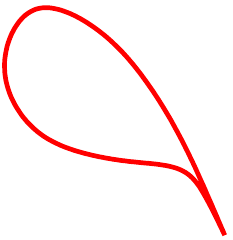} & \includegraphics[scale=1]{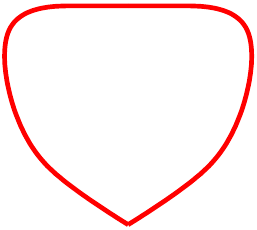} &  \includegraphics[scale=1]{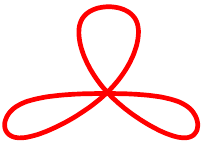}\tabularnewline
$A_5^*$  & $A_6$ & $E_6$ & $D_4^*$\tabularnewline\tabularnewline
\multicolumn{4}{c}{One singular point}\tabularnewline
\end{tabular}

\bigskip

\caption{Real projective quartics associated with rational smooth embeddings of $\mathbb{A}_{\mathbb{R}}^{1}$.}
\label{QuarticSing-RatEmb}
\end{table}

\begin{proof}
1) Assume that $C$ has three singular points. The genus formula implies they all have multiplicity sequence equal to $[2]$. At least one of these points, say $p_\infty$ is real and hence necessarily contained in $C\cap L_\infty$. If the other two were real, then since they must also be contained in $C\cap L_\infty$, we would have $(C \cdot L_\infty )\geq 6$ which is absurd. The other two singular 
points are thus a pair $\{p,\overline{p}\}$ of non-real complex conjugate points of type $2A_1^i$ or $2A_2^i$ of $C$. Furthermore, neither $p$ nor $\overline{p}$ belongs to $C\cap L_\infty$. We now discuss the possibilities according to the type of the point $p_\infty$. 

\quad a) $p_\infty$ is an ordinary double point. Then $p_\infty$ is necessarily of type $A_1^*$ for otherwise its inverse image in the normalization of $C$ would consists of two real point, in contradiction with Lemma \ref{lem:RatEmbed-from-projcurve}. This implies in particular that $L_\infty$ is not tangent to $C$ at $p_\infty$, hence that the intersection multiplicity of $C$ and $L_\infty$ at $p_\infty$ is equal to $2$. Since the inverse image of $C\cap L_\infty$ in the normalization of $C$ contains a unique real point and $C\cdot L_\infty=4$, we conclude that $C\cap (L_\infty\setminus p_\infty)$ consists of a unique real point $p_s$ and that $L_\infty$ is the tangent to $C$ at $p_s$.  This yields the possible types $A_{1}^*+2A_{1}^{i}$ and $A_{1}^*+2A_{2}^{i}$. To exclude the type $A_{1}^*+2A_{2}^{i}$, we observe that 
if it existed then the composition of the normalization morphism $\nu : \mathbb{P}^1_{\mathbb{R}}\rightarrow C$ with the morphism $C\rightarrow \mathbb{P}^1_{\mathbb{R}}$ induced by the linear projection from the point $p_\infty$ would be a double cover $\mathbb{P}^1_{\mathbb{R}}\rightarrow \mathbb{P}^1_{\mathbb{R}}$ ramified at the inverse images of $p_s$ and the pair of  non-real complex conjugate points $p$ and $\overline{p}$, which is impossible. We thus obtain case a).

\quad b) $p_\infty$ is an ordinary cusp $A_2$. Then the tangent line to $C$ at $p_{\infty}$ is distinct from $L_{\infty}$.
Indeed, otherwise we would have $(L_{\infty}\cdot C)_{p_{\infty}}=3$
and then $L_{\infty}\cap C$ would consists of $p_{\infty}$ and another
smooth real point of $C$ at which $C$ and $L_{\infty}$ intersects
transversally in contradiction with Lemma \ref{lem:RatEmbed-from-projcurve}.
So $(L_{\infty}\cdot C)_{p_{\infty}}=2$ and since the inverse image of $C\cap L_\infty$ already contains the inverse image of $p_s$ as a real point, it follows again from Lemma \ref{lem:RatEmbed-from-projcurve} 
that $\left(C\cap L_{\infty}\right)\setminus\left\{ p_{\infty}\right\}$ consists of a pair $\{q,\overline{q}\}$ of non-real complex conjugate
smooth points $q$ and $\overline{q}$ of $C$. This yields the possible types $A_{2}+2A_{2}^{i}$ and $A_{2}+2A_{1}^{i}$.

2) Assume that $C$ has two singular points, say $p_{0}$ and $p_{1}$. Then the genus formula
implies that these points have respective multiplicity sequences
$[2]$ and $[2,2]$. Since these are different, it follows that
$p_{0}$ and $p_{1}$ are real points of $C$, which are therefore
supported on $L_{\infty}$. Furthermore, since $C\cdot L_{\infty}=4$,
we have $C\cap L_{\infty}=\{p_{0},p_{1}\}$ and $L_\infty$ is not tangent to $C$ at $p_0$ or $p_1$. Since the inverse image of $C\cap L_\infty$ in the normalization of $C$ contains a unique real point, the union of the branches of $C$ at $p_0$ and $p_1$ contains a unique real branch. It follows that either $p_0$ is of type $A_2$ and then $p_1$ must be of type $A_3^*$ or $p_0$ is of type $A_1^*$ and then $p_1$ must be of type $A_4$.

3) Assume that  $C$ has a unique singular point $p_\infty$. Then $p_\infty$ is a real point of $C$, which must then be contained  $C\cap L_{\infty}$. By the genus formula, the possible multiplicity sequences for $p_{\infty}$ are $[3]$ or $[2,2,2]$. 

If $p_{\infty}$ is a triple point, then it follows from Lemma \ref{lem:RatEmbed-from-projcurve}
that $p_{\infty}$ is either a cusp of multiplicity three $E_{6}$
or an ordinary triple point of type $D_{4}^{*}$. Indeed, in the other two cases $D_5$ and $D_4$, the inverse image of $p_\infty$ in the normalization of $C$ would consist of two and three real points respectively. Since $(C\cdot L_\infty)_{p_{\infty}}\geq 3$, we conclude that  $L_{\infty}$ is a tangent line to $C$ at $p_{\infty}$. Indeed otherwise, if  $(C\cdot L_\infty)_{p_{\infty}}=3$ then $L_{\infty}$ would also intersect $C$ transversally in a smooth real point, which is again impossible by Lemma \ref{lem:RatEmbed-from-projcurve}.

Otherwise, if $p_{\infty}$ has multiplicity sequence $[2,2,2]$ then $p_{\infty}$ is either
a ramphoid cusp $A_{6}$ or an oscnode, which is then necessarily of type $A_{5}^{*}$ by Lemma \ref{lem:RatEmbed-from-projcurve} again. In the second case, the tangent line to $C$ at $p_\infty$ intersects $C$ with multiplicity four at $p_\infty$, so it cannot be equal to $L_\infty$. Indeed, otherwise we would have $C\cap L_\infty=\{p_\infty\}$ and then the inverse image of $C\cap L_\infty$ in the normalization of $C$ would not contain any real point, in contradiction to Lemma \ref{lem:RatEmbed-from-projcurve}.   
\end{proof}

\begin{prop}
\label{prop:rectif-quartic} Every rational smooth embedding
$f:\mathbb{A}_{\mathbb{R}}^{1}\dashrightarrow\mathbb{A}_{\mathbb{R}}^{2}$
whose associated projective curve $C$ is a quartic is biddable.
Furthermore, it is rectifiable except possibly in the cases where $C$ has
either a singular point of type $A_{5}^{*}$ or a singular point of
type $A_{6}$ whose tangent is different from $L_{\infty}$. 
\end{prop}

To prove Proposition \ref{prop:rectif-quartic}, we consider below
the different possible cases for $C$ separately, according to the
number of its singular points. 

\subsubsection{Three singular points}

\begin{lem}
\label{lem:Rectif-3Sing-bis}A rational smooth embedding
$f:\mathbb{A}_{\mathbb{R}}^{1}\dashrightarrow\mathbb{A}_{\mathbb{R}}^{2}$
whose associated curve $C\subset\mathbb{P}_{\mathbb{R}}^{2}$ is a
quartic with singularities $A_{1}^*+2A_{1}^{i}$ is rectifiable.
\end{lem}

\begin{proof}
The curve $C$ has a real singular point $p_{\infty}\in C\cap L_{\infty}$
of type $A_{1}^{*}$ with pair of non-real complex conjugate tangents
$T_{\infty}$ and $\overline{T}_{\infty}$, and $L_{\infty}\cap C=\{p_{\infty},p_{0}\}$
where $p_{0}$ is a smooth real point of $C$, at which $C$ intersects
$L_{\infty}$ with multiplicity $2$. We denote by $p$ and $\overline{p}$
the pair of non-real complex conjugate singular points of $C$. Since
$T_{\infty}$ intersects $C$ at $p_{\infty}$ with multiplicity $3$,
it also intersects $C$ transversally at a point $q$ different from
$p$ and $\overline{p}$, and $\overline{T}_{\infty}$ then intersects
$C$ transversally at $\overline{q}$. We let $F=[p_{\infty},p]$
and $\overline{F}=[p_{\infty},\overline{p}]$ be the pair of non-real
complex conjugate lines through $p_{\infty}$ and $p$ and $p_{\infty}$
and $\overline{p}$ respectively. Note that the pairs of lines $\{T_{\infty},\overline{T}_{\infty}\}$
and $\{F,\overline{F}\}$ are different. Let $\tau_{1}:V_{1}=\mathbb{F}_{1}\rightarrow\mathbb{P}_{\mathbb{R}}^{2}$
be the blow-up the point $p_{\infty}$ with exceptional divisor $E_{\infty}$,
and let $\rho_{1}:V_{1}=\mathbb{F}_{1}\rightarrow\mathbb{P}^{1}$
be the induced $\mathbb{P}^{1}$-bundle structure with exceptional
section $E_{\infty}$. The proper transform of $C$ intersects $E_{\infty}$
transversally in a pair of non-real complex conjugate points $p_{\infty,1}$
and $\overline{p}_{\infty,1}$. The proper transforms of $T_{\infty}$
and $\overline{T}_{\infty}$ are non-real complex conjugate fibers
of $\rho_{1}$ which intersect $E_{\infty}$ 
transversally in $p_{\infty,1}$ and $\overline{p}_{\infty,1}$,
respectively, and also intersect the proper transform of $C$ at the
points $q$ and $\overline{q}$ respectively. Let $\tau_{2}:V_{1}\dashrightarrow 
V_{2}
=\mathbb{F}_{1}$
be the birational map defined over $\mathbb{R}$ consisting of the
blow-up of $p_{\infty,1}$, $\overline{p}_{\infty,1}$, $p$ and $\overline{p}$
followed by the contraction of the proper transforms of $T_{\infty}$,
$\overline{T}_{\infty}$, $F$ and $\overline{F}$ (See Figure \ref{fig:threesing-forgotten}).

\begin{figure}
\input{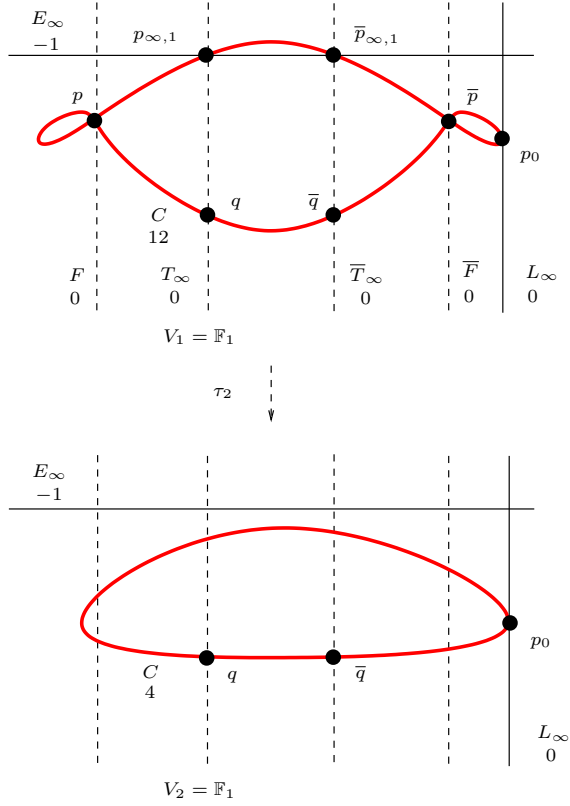} 
\caption{
Three singular points $A_{1}^*+2A_{1}^{i}$.}  
\label{fig:threesing-forgotten}
\end{figure}

The proper transform of $E_{\infty}$ is a section of the $\mathbb{P}^{1}$-bundle
structure $\rho'_{1}:\mathbb{F}_{1}\rightarrow\mathbb{P}^{1}$ with
self-intersection $-1$ while the proper transform of $C$ is a smooth
$2$-section of $\rho_{1}'$ with self-intersection $4$ disjoint
from $E_{\infty}$. The images of $C$ and $L_{\infty}$ by the contraction
$\tau_{3}:V_{2}\rightarrow\mathbb{P}_{\mathbb{R}}^{2}$ of $E_{\infty}$
are then a smooth conic $C'$ and its tangent line $L_{\infty}'$
at the real point $p_{0}$. The composition $\tau_{3}\circ\tau_{2}\circ\tau_{1}^{-1}:\mathbb{P}_{\mathbb{R}}^{2}\dashrightarrow\mathbb{P}_{\mathbb{R}}^{2}$
is a birational diffeomorphism which maps $C\cup L_{\infty}$ onto
$C'\cup L_{\infty}'$ and restricts to a birational diffeomorphism
$\alpha:\mathbb{A}_{\mathbb{R}}^{2}=\mathbb{P}_{\mathbb{R}}^{2}\setminus L_{\infty}\dashrightarrow\mathbb{P}_{\mathbb{R}}^{2}\setminus L_{\infty}'=\mathbb{A}_{\mathbb{R}}^{2}$.
The composition $\alpha\circ f:\mathbb{A}_{\mathbb{R}}^{1}\dashrightarrow\mathbb{A}_{\mathbb{R}}^{2}$
is then rectifiable by Proposition \ref{prop:rectif-up3} and hence, $f$ is rectifiable.
\end{proof}

\begin{lem}
\label{lem:Rectif-3Sing}A rational smooth embedding
$f:\mathbb{A}_{\mathbb{R}}^{1}\dashrightarrow\mathbb{A}_{\mathbb{R}}^{2}$
whose associated curve $C\subset\mathbb{P}_{\mathbb{R}}^{2}$ is a
quartic with singularities $A_{2}+2A_{1}^{i}$ or $A_{2}+2A_{2}^{i}$
is rectifiable.
\end{lem}

\begin{proof}
In both cases, $C$ has a unique real ordinary cusp $p_{\infty}\in C\cap L_\infty$
whose tangent line is different from $L_\infty$. So $(C\cdot L_\infty)_{p_\infty}=2$ and then,
by Lemma \ref{lem:RatEmbed-from-projcurve}, $(C\cap L_{\infty})\setminus\{p_{\infty}\}$  
consists of a pair $\{q,\overline{q}\}$ of non-real complex conjugate
smooth points $q$ and $\overline{q}$ of $C$. We denote by $\{p,\overline{p}\}$
the pair of non-real complex conjugate singular points of $C$. To
show that $f:\mathbb{A}_{\mathbb{R}}^{1}\dashrightarrow\mathbb{A}_{\mathbb{R}}^{2}$
is rectifiable, we proceed as follows. 

Step 1) (See Figure \ref{fig:threesing1}) Let $\tau_{1}:\mathbb{F}_{1}\rightarrow\mathbb{P}_{\mathbb{R}}^{2}$
be the blow-up the point $p_{\infty}$ with exceptional divisor $E_{\infty}$,
and let $\rho_{1}:V_{1}=\mathbb{F}_{1}\rightarrow\mathbb{P}^{1}$
be the induced $\mathbb{P}^{1}$-bundle structure with exceptional
section $E_{\infty}$. The proper transforms of the non-real conjugate
lines $L=[p_{\infty},p]$ and $\overline{L}=[p_{\infty},\overline{p}]$
are fibers of $\rho_{1}$ which intersect the proper transform of
$C$ in the double points $p$ and $\overline{p}$ respectively. On
the other hand $C$ intersects the proper transform of $L_{\infty}$
transversally in the two conjugate points $q$ and $\overline{q}$,
and it intersects $E_{\infty}$ with multiplicity $2$ at a smooth
real point $p_{\infty,1}$ supported on $E_{\infty}\setminus\left\{ L_{\infty},L,\overline{L}\right\} $.
Let $\tau_{2}:V_{1}\dashrightarrow V_{2}=\mathbb{F}_{1}$ be the birational
map defined over $\mathbb{R}$ consisting of the blow-up of the pair
of points $p$ and $\overline{p}$ followed by the contraction of
the proper transforms of $L$ and $\overline{L}$. The proper transform
of $E_{\infty}$ is a section of the $\mathbb{P}^{1}$-bundle structure
$\rho_{1}':\mathbb{F}_{1}\rightarrow\mathbb{P}^{1}$ with self-intersection
$1$ while the proper transform of $C$ is a smooth $2$-section of
$\rho_{1}'$ with self-intersection $4$, still intersecting $E_{\infty}$
with multiplicity $2$ in $p_{\infty,1}$. 

\begin{figure}
\input{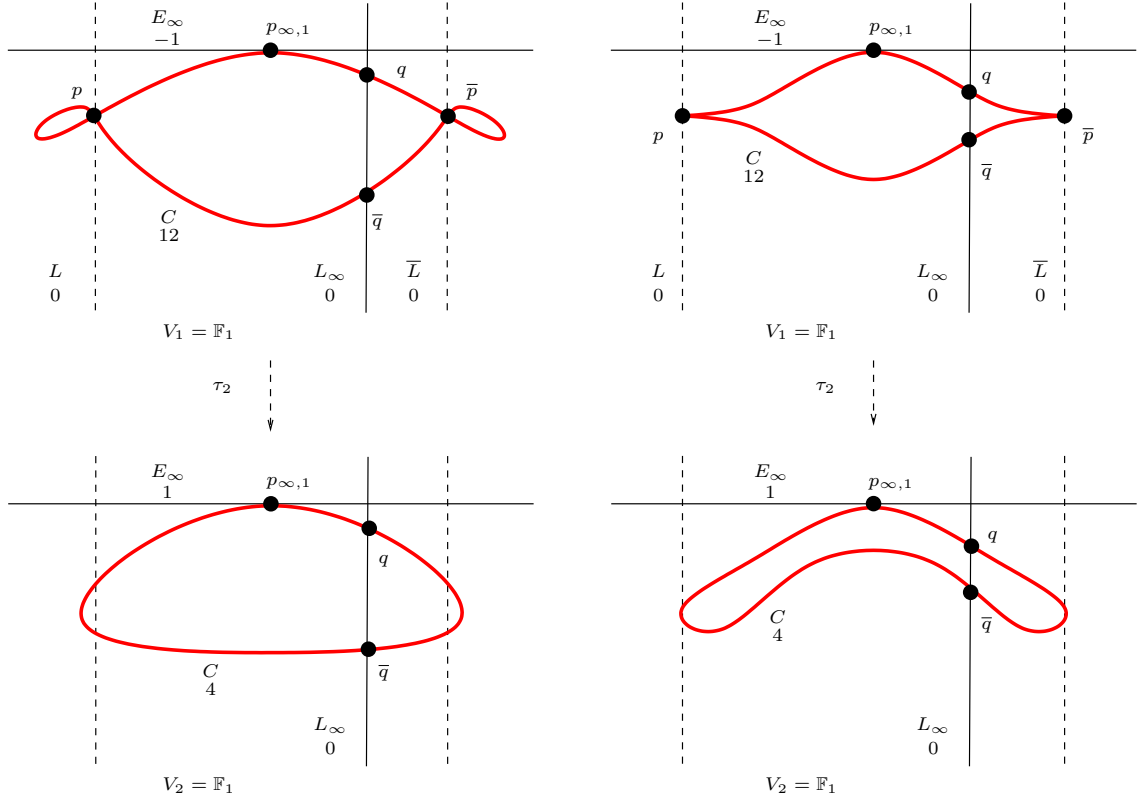} 
\caption{Three singular points, step 1: case $A_{2}+2A_{1}^{i}$ on the left and $A_{2}+2A_{2}^{i}$ on the right.}  
\label{fig:threesing1}
\end{figure}

Step 2) The proper transforms $Q$ and $\overline{Q}$ in $V_{2}$
of the pair of non-real complex conjugate conics in $\mathbb{P}_{\mathbb{R}}^{2}$
passing through $p_{\infty},p_{\infty,1},q,p,\overline{p}$ and $p_{\infty},p_{\infty,1},\overline{q},p,\overline{p}$
respectively are sections of $\rho_{1}'$ passing through $q$ and
$p_{\infty,1}$ and $\overline{q}$ and $p_{\infty,1}$ respectively.
Let $\tau_{3}:V_{2}\dashrightarrow V_{3}$ be the birational map defined
over $\mathbb{R}$ consisting of the blow-up of $q$ and $\overline{q}$,
the blow-up of the $\mathbb{R}$-rational point $p_{\infty,1}$ with
exceptional divisor $E_{\infty,1}$, followed by the contraction of
the proper transforms of $Q$ and $\overline{Q}$. The proper transforms
in $V_{3}$ of $L_{\infty}$, $E_{\infty}$, $E_{\infty,1}$ have
self-intersections $-2$, $0$ and $1$ respectively, the proper transform
of $C$ has self-intersection $1$, it intersects $E_{\infty}$ transversally
at the real point $p_{\infty,2}=E_{\infty}\cap E_{\infty,1}$ and
it does no longer intersect $L_{\infty}$. 

\begin{figure}
\input{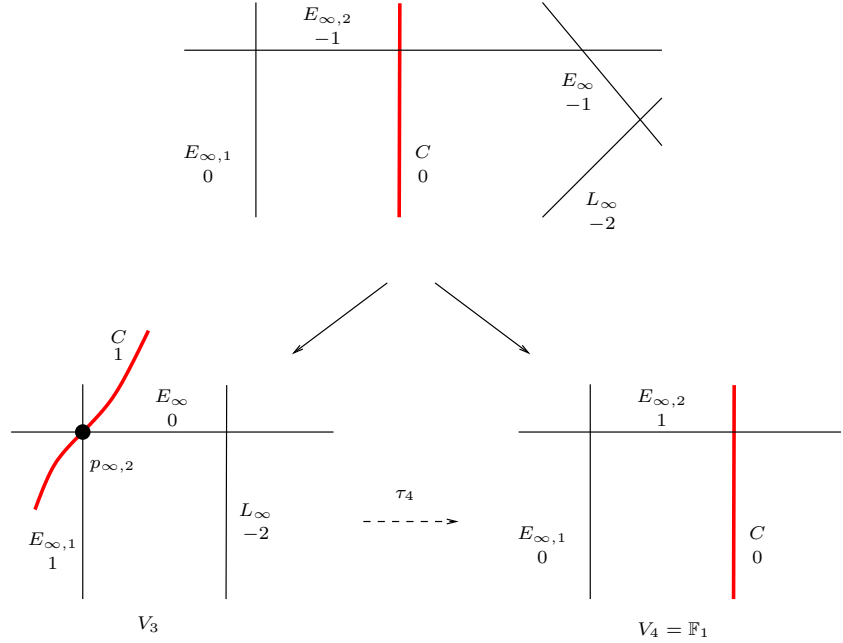} 
\caption{Three singular points, step 3: case $A_{2}+2A_{1}^{i}$ or $A_{2}+2A_{2}^{i}$.}  
\label{fig:threesing3}
\end{figure}

Step 3) (See Figure \ref{fig:threesing3}) Finally, let $\tau_{4}:V_{3}\dashrightarrow V_{4}$ be the
birational map defined over $\mathbb{R}$ obtained by blowing up $p_{\infty,2}$
with exceptional divisor $E_{\infty,2}$ and then contracting successively
the proper transforms of $E_{\infty}$ and $L_{\infty}$. The proper
transforms of $E_{\infty,1}$ and $C$ have self-intersection $0$ while the proper transform of
$E_{\infty,2}$ has self-intersection $1$.  Counting the number of points blown-up and curves contracted,
we see that $\mathrm{Pic}(V_{4})\simeq\mathbb{Z}^{2}$. It follows
that $V_{4}\simeq\mathbb{F}_{1}$ on which $C$ and $E_{\infty,1}$
and $E_{\infty,2}$ are respectively fibers and a section of the $\mathbb{P}^1$-bundle structure $\rho_{1}:\mathbb{F}_{1}\rightarrow\mathbb{P}_{\mathbb{R}}^{1}$.
By construction, the composition $\tau_{4}\circ\tau_{3}\circ\tau_{2}\circ\tau_{1}^{-1}:\mathbb{P}_{\mathbb{R}}^{2}\dashrightarrow\mathbb{F}_{1}$
restricts to a birational diffeomorphism 
\[
\alpha:\mathbb{A}_{\mathbb{R}}^{2}=\mathbb{P}_{\mathbb{R}}^{2}\setminus L_{\infty}\dashrightarrow\mathbb{F}_{1}\setminus(E_{\infty,1}\cup E_{\infty,2})\simeq\mathbb{A}_{\mathbb{R}}^{2}
\]
such that $\alpha\circ f:\mathbb{A}_{\mathbb{R}}^{1}\dashrightarrow\mathbb{A}_{\mathbb{R}}^{2}$
coincides with the closed immersion of $\mathbb{A}_{\mathbb{R}}^{1}$
as a real fiber of the projection $\mathrm{pr}_{2}=\rho_1|_{\mathbb{A}_{\mathbb{R}}^{2}}:\mathbb{A}_{\mathbb{R}}^{2}\rightarrow\mathbb{A}_{\mathbb{R}}^{1}$.
This shows that $f$ is rectifiable. 
\end{proof}

\subsubsection{Two singular points }
\begin{lem}
A rational smooth embedding $f:\mathbb{A}_{\mathbb{R}}^{1}\dashrightarrow\mathbb{A}_{\mathbb{R}}^{2}$
whose associated curve $C\subset\mathbb{P}_{\mathbb{R}}^{2}$ is a
quartic with singularities $A_{2}+A_{3}^{*}$ or $A_{4}+A_{1}^{*}$
is rectifiable. 
\end{lem}

\begin{proof}
We deal with each case separately (see Figure \ref{fig:twosing}).

1) $C$ has a real ordinary cusp $p_{0}$ and a real tacnode $p_{1}$ with real tangent, whose inverse image in the normalization of $C$ consists of a pair $(q,\overline{q})$ of non-real complex conjugate points. Furthermore the tangent to $C$ at $p_0$ and $p_1$ is different from $L_\infty$. By Lemma \ref{lem:RatEmbed-from-projcurve}, we have $C\cap L_{\infty}=\{p_{0},p_{1}\}$. 
A minimal log-resolution $\tau:(V,B)\rightarrow(\mathbb{P}_{\mathbb{R}}^{2},C\cup L_{\infty})$
of the pair $(\mathbb{P}_{\mathbb{R}}^{2},C\cup L_{\infty})$ for which $B$ has no imaginary loop is obtained
by blowing-up on the one hand the real point $p_{0}$ with exceptional
divisor $E_{0,1}$, then the real intersection point $p_{0,1}\neq L_\infty\cap E_{0,1}$ of
the proper transform of $C$ with $E_{0,1}$, with exceptional divisor
$E_{0,2}$ and then the real intersection point $p_{0,2}$ of the
proper transform of $C$ with $E_{0,2}$, with exceptional divisor
$E_{0,3}$, and on the other hand the real point $p_{1}$ with exceptional
divisor $E_{1,1}$, the real intersection point $p_{1,1}\neq L_\infty\cap E_{1,1}$ of the
proper transform of $C$ with $E_{1,1}$, with exceptional divisor
$E_{1,2}$ and then the pair of non-real complex conjugate intersection
points $(q,\overline{q})$ of the proper transform of $C$ with $E_{1,2}$,
with respective exceptional divisors $F$ and $\overline{F}$. The
proper transform of $C$ in $V$ is a smooth rational curve with self-intersection
$0$, which intersects $E_{0,3}$, $F$ and $\overline{F}$ transversally.
The proper transforms in $V$ of the non-real complex conjugate smooth
conics $Q$ and $\overline{Q}$ passing through $p_{0},p_{0,1},p_{1},p_{1,1},q$
and $p_{0},p_{0,1},p_{1},p_{1,1},\overline{q}$ respectively are disjoint
$(-1)$-curves which do not intersect the proper transform of $C$.
The contraction $\tau_{1}:V\rightarrow V_{1}$ of $Q\cup\overline{Q}\cup L_{\infty}\cup E_{1,1}\cup E_{0,1}\cup E_{1,2}$ is a birational morphism defined over $\mathbb{R}$. The images in
$V_{1}$ of $C$ and $E_{0,2}$ have self-intersection $0$ while the image of $E_{0,3}$ has self-intersection
$1$.  We conclude in the same way as in the proof of Lemma \ref{lem:Rectif-3Sing}
that $V_{1}\simeq\mathbb{F}_{1}$ on which $C$, $E_{0,2}$ and $E_{0,3}$
are respectively fibers and a section of $\rho_{1}:\mathbb{F}_{1}\rightarrow\mathbb{P}_{\mathbb{R}}^{1}$,
that $\tau_{1}\circ\tau^{-1}:\mathbb{P}_{\mathbb{R}}^{2}\dashrightarrow\mathbb{F}_{1}$
restricts to a birational diffeomorphism 
\[
\alpha:\mathbb{A}_{\mathbb{R}}^{2}=\mathbb{P}_{\mathbb{R}}^{2}\setminus L_{\infty}\dashrightarrow\mathbb{F}_{1}\setminus(E_{0,3}\cup E_{0,2})\simeq\mathbb{A}_{\mathbb{R}}^{2}
\]
for which $\alpha\circ f:\mathbb{A}_{\mathbb{R}}^{1}\dashrightarrow\mathbb{A}_{\mathbb{R}}^{2}$
is a closed embedding of $\mathbb{A}_{\mathbb{R}}^{1}$
as a real fiber of the projection $\mathrm{pr}_{2}:\mathbb{A}_{\mathbb{R}}^{2}\rightarrow\mathbb{A}_{\mathbb{R}}^{1}$. 

\begin{figure}
\input{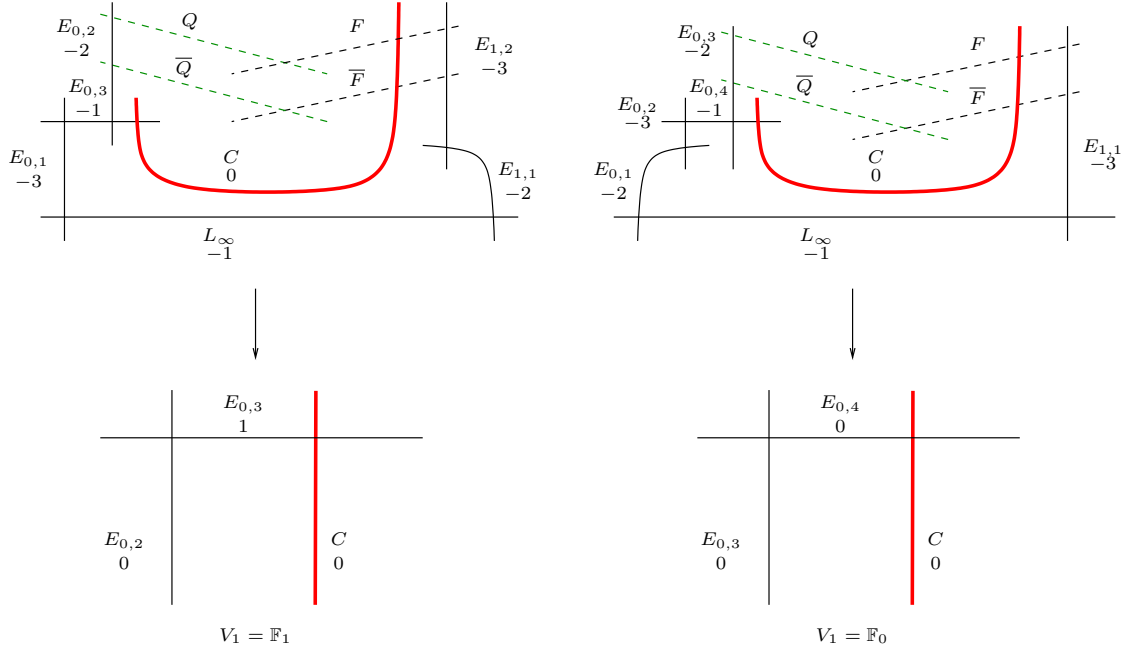} 
\caption{Two singular points: case $A_{2}+A_{3}^{*}$ on the left and $A_{4}+A_{1}^{*}$ on the right.}  
\label{fig:twosing}
\end{figure}

2) $C$ has real double cusp $p_{0}$ and a real ordinary double point
$p_{1}$ with non-real complex conjugate tangents. Furthermore, the tangent line to $C$ at each of them is different from $L_\infty$.
As in the previous case, we have $C\cap L_{\infty}=\{p_{0},p_{1}\}$.  A minimal log-resolution
$\tau:(V,B)\rightarrow(\mathbb{P}_{\mathbb{R}}^{2},C\cup L_{\infty})$
of the pair $(\mathbb{P}_{\mathbb{R}}^{2},C\cup L_{\infty})$ for
which $B$ has no imaginary loop is obtained by blowing-up on the
one hand the real point $p_{0}$ with exceptional divisor $E_{0,1}$,
then the real intersection point $p_{0,1}\neq E_{0,1}\cap L_{\infty}$ of the proper transform
of $C$ with $E_{0,1}$, with exceptional divisor $E_{0,2}$, then
the real intersection point $p_{0,2}$ of the proper transform of
$C$ with $E_{0,2}$ with exceptional divisor $E_{0,3}$ and then
the real intersection point $p_{0,3}$ of the proper transform of
$C$ with $E_{0,3}$ with exceptional divisor $E_{0,4}$, and on the
other the real point $p_{1}$ with exceptional divisor $E_{1,1}$
and then the pair of non-real complex conjugate intersection points
$(q,\overline{q})$ of the proper transform of $C$ with $E_{1,1}$,
with respective exceptional divisors $F$ and $\overline{F}$. The
proper transform of $C$ in $V$ is a smooth rational curve with self-intersection
$0$, which intersects $E_{0,4}$, $F$ and $\overline{F}$ transversally.
The proper transforms in $V$ of the non-real complex conjugate smooth
conics $Q$ and $\overline{Q}$ passing through $p_{0},p_{0,1},p_{0,2},p_{1},q$
and $p_{0},p_{0,1},p_{0,2},p_{1},\overline{q}$ respectively are disjoint
$(-1)$-curves which do not intersect the proper transform of $C$.
The contraction $\tau_{1}:V\rightarrow V_{1}$ of $Q\cup\overline{Q}\cup L_{\infty}\cup E_{0,1}\cup E_{1,1}\cup E_{0,2}$ is a birational morphism defined over $\mathbb{R}$. The images in
$V_{1}$ of $C$, $E_{0,3}$ and $E_{0,4}$ all have self-intersection
$0$. Counting the number of points blown-up and curves contracted, we conclude that $V_{1}\simeq\mathbb{F}_{0}=\mathbb{P}_{\mathbb{R}}^{1}\times\mathbb{P}_{\mathbb{R}}^{1}$ in which the images of $C$, $E_{0,3}$ and $E_{0,4}$ are respectively fibers and a section of $\mathrm{pr}_{2}:\mathbb{F}_{0}\rightarrow\mathbb{P}_{\mathbb{R}}^{1}$,
and that 
\[
\alpha=\tau_{1}\circ\tau^{-1}|_{\mathbb{P}_{\mathbb{R}}^{2}\setminus L_{\infty}}:\mathbb{A}_{\mathbb{R}}^{2}=\mathbb{P}_{\mathbb{R}}^{2}\setminus L_{\infty}\dashrightarrow\mathbb{F}_{0}\setminus(E_{0,4}\cup E_{0,3})\simeq\mathbb{A}_{\mathbb{R}}^{2}
\]
is a birational diffeomorphism for which $\alpha\circ f:\mathbb{A}_{\mathbb{R}}^{1}\dashrightarrow\mathbb{A}_{\mathbb{R}}^{2}$
is a closed embedding of $\mathbb{A}_{\mathbb{R}}^{1}$
as a real fiber of the projection $\mathrm{pr}_{2}:\mathbb{A}_{\mathbb{R}}^{2}\rightarrow\mathbb{A}_{\mathbb{R}}^{1}$. 
\end{proof}

\subsubsection{A unique singular point}
\begin{lem} \label{lem:unique-sing}
Let $f:\mathbb{A}_{\mathbb{R}}^{1}\dashrightarrow\mathbb{A}_{\mathbb{R}}^{2}$
be a rational smooth embedding whose associated curve
$C\subset\mathbb{P}_{\mathbb{R}}^{2}$ is a quartic with a unique
singular point $p_{\infty}$. Then the following holds: 

a) If $p_{\infty}$ is an oscnode $A_{5}^{*}$ or a ramphoid
cusp $A_{6}$ such that $T_{p_{\infty}}C\neq L_{\infty}$ then $f$
is biddable. 

b) If $p_{\infty}$ is either triple point $D_{4}^{*}$ or $E_{6}$ or
a ramphoid cusp $A_{6}$ such that $T_{p_{\infty}}C=L_{\infty}$ then
$f$ is rectifiable.  
\end{lem}

\begin{proof}
We consider each of the possible singularities $E_6$, $D_4^*$, $A_6$ and $A_5^*$ listed in Lemma \ref{lem:SingTypes} separately. We begin with the curves listed in case b). 

1) If $p_{\infty}$ is a multiplicity $3$ cusp $E_{6}$ or a ramphoid cusp $A_{6}$ with tangent $T_{p_{\infty}}C=L_{\infty}$, then $C\cap L_\infty=p_\infty$ and hence $f:\mathbb{A}_{\mathbb{R}}^{1}\dashrightarrow\mathbb{A}_{\mathbb{R}}^{2}$ is  an everywhere defined closed embedding, which is thus rectifiable by the Abhyankar-Moh Theorem over $\mathbb{R}$. 
\newline

2) If $p_{\infty}$ is a real triple point of type $D_{4}^{*}$ then
the tangent cone to $C$ at $p_{\infty}$ is the union of $L_{\infty}$
and a pair $L$ and $\overline{L}$ of non-real complex conjugate
tangents to $C$. Since every line in $\mathbb{P}_{\mathbb{R}}^{2}$
passing through $p_{\infty}$ intersects $C\setminus\{p_{\infty}\}$
transversally in at most one real point, the proper transform of $C$
in the blow-up $\tau_{1}:\mathbb{F}_{1}\rightarrow\mathbb{P}_{\mathbb{R}}^{2}$
of $p_{\infty}$ with exceptional divisor $E\simeq\mathbb{P}_{\mathbb{R}}^{1}$
is a section of the $\mathbb{P}^{1}$-bundle structure $\rho_{1}:\mathbb{F}_{1}\rightarrow\mathbb{P}_{\mathbb{R}}^{1}$.
It intersects $E$ transversally at the pair of non-real complex conjugate
intersection points $q$ and $\overline{q}$ of the proper transforms
of $L$ and $\overline{L}$ with $E$ respectively and at the real
intersection point of the proper transform of $L_{\infty}$ with $E$.
Let $\tau_{2}:\mathbb{F}_{1}\dashrightarrow\mathbb{F}_{3}$ be the
birational map defined over $\mathbb{R}$ consisting of the blow-up
of $q$ and $\overline{q}$ followed by the contraction of the proper
transform of $L$ and $\overline{L}$. The proper transforms of $L_{\infty}$
and $E$ are respectively a fiber and the negative section of the
$\mathbb{P}^{1}$-bundle structure $\rho_{3}:\mathbb{F}_{3}\rightarrow\mathbb{P}_{\mathbb{R}}^{1}$
while the proper transform of $C$ in $\mathbb{F}_{3}$ is a section
of $\rho_{3}$ intersecting $E$ transversally at the real point $E\cap L_{\infty}$.
The birational map $\tau_{2}\circ\tau_{1}^{-1}$ restricts to a birational
diffeomorphism $\alpha:\mathbb{A}_{\mathbb{R}}^{2}=\mathbb{P}_{\mathbb{R}}^{2}\setminus L_{\infty}\stackrel{\sim}{\dashrightarrow}\mathbb{F}_{3}\setminus(E\cup L_{\infty})\simeq\mathbb{A}_{\mathbb{R}}^{2}$
for which the composition $\alpha\circ f:\mathbb{A}_{\mathbb{R}}^{1}\dashrightarrow\mathbb{A}_{\mathbb{R}}^{2}$
is an everywhere defined closed embedding. By the Abhyankar-Moh Theorem
over $\mathbb{R}$, $\alpha\circ f$ is rectifiable, and so $f$
is rectifiable. 
\newline

3) If $p_{\infty}$ is a ramphoid cusp $A_{6}$ such that $T_{p_{\infty}}C\neq L_{\infty}$
then $C\cap(L_{\infty}\setminus\{p_{\infty}\})$ consists of a pair
$\{q,\overline{q}\}$ of non-real complex conjugate points while the
tangent $T$ to $C$ at $p_{\infty}$ intersects $C$ with multiplicity
$4$ at $p_{\infty}$. Let $\overline{\rho}:\mathbb{P}_{\mathbb{R}}^{2}\dashrightarrow\mathbb{P}_{\mathbb{R}}^{1}$
be the rational map defined by the pencil $\mathcal{M}\subset|\mathcal{O}_{\mathbb{P}_{\mathbb{R}}^{2}}(4)|$
generated by $C$ and $3T+L_{\infty}$. A minimal resolution $\tau:V_{1}\rightarrow\mathbb{P}_{\mathbb{R}}^{2}$
of the indeterminacy of $\overline{\rho}$ is obtained by blowing-up
$q$ and $\overline{q}$ with respective non-real complex conjugate
exceptional divisor $F$ and $\overline{F}$, the point $p_{\infty}$
with exceptional divisor $E_{1}\simeq\mathbb{P}_{\mathbb{R}}^{1}$
and four times successively the intersection point $p_{\infty,i}$
of the proper transform of $C$ and $E_{i}$, with exceptional divisor
$E_{i+1}\simeq\mathbb{P}_{\mathbb{R}}^{1}$. The reduced total transform
$B_{1}$ of $L_{\infty}$ is a tree of smooth geometrically rational
curves. The proper transform of $C$ in $V$ is a smooth rational
curve with self-intersection $0$, and so $\overline{\pi}_{1}=\overline{\rho}\circ\tau:V_{1}\rightarrow\mathbb{P}_{\mathbb{R}}^{1}$
is a $\mathbb{P}^{1}$-fibration, having the proper transform of $C$
as a smooth real fiber, and the exceptional divisors $F$, $\overline{F}$
and $E_{5}$ as sections (See left hand side of Figure \ref{fig:onesing}). 

The fiber of $\overline{\pi}_{1}$ over the real point $\overline{\pi}_{1}(E_{4})$
consists of the union of $E_{4}$ and the proper transforms of the
non-real complex conjugate smooth conics $Q$ and $\overline{Q}$
passing respectively through $p_{\infty},p_{\infty,1},p_{\infty,2},p_{\infty,3},q$
and $p_{\infty},p_{\infty,1},p_{\infty,2},p_{\infty,3},\overline{q}$.
These are disjoint $(-1)$-curves in $V_{1,\mathbb{C}}$ which can
thus be contracted by a birational morphism $\sigma:V_{1}\rightarrow V$
defined over $\mathbb{R}$. We let $B=L_{\infty}\cup E_{1}\cup\cdots\cup E_{5}$
be the proper transform of $B_{1,\mathbb{R}}$. The proper transform
of $E_{4}$ is then a full fiber of the induced $\mathbb{P}^{1}$-fibration
$\overline{\pi}=\overline{\pi}_{1}\circ\sigma^{-1}:V\rightarrow\mathbb{P}_{\mathbb{R}}^{1}$
and the latter restricts on $S=V\setminus B$ to an $\mathbb{A}^{1}$-fibration
$\pi:S\rightarrow\mathbb{A}_{\mathbb{R}}^{1}$ with a unique singular
fiber equal to $T_{p_{\infty}}\cap S\simeq\mathbb{A}_{\mathbb{R}}^{1}$
with multiplicity $3$. By Theorem \ref{thm:-A1Fib-Model}, $S$ is an 
algebraic model of $\mathbb{R}^{2}$ and by construction, $\alpha=\sigma\circ\tau^{-1}:\mathbb{P}_{\mathbb{R}}^{2}\dashrightarrow V$
induces a birational diffeomorphism $\mathbb{P}_{\mathbb{R}}^{2}\setminus L_{\infty}\dashrightarrow S$
that maps $C$ to an irreducible fiber of $\overline{\pi}$. This
shows that $f:\mathbb{A}_{\mathbb{R}}^{1}\dashrightarrow\mathbb{A}_{\mathbb{R}}^{2}$
is biddable. Note that since $\pi$ has $T\cap S\simeq \mathbb{A}^1_{\mathbb{R}}$
as a non-reduced fiber, none of the sufficient conditions for rectifiablity given Proposition \ref{thm:Main-rectif-criterion} is satisfied.   
\newline

\begin{figure}
\input{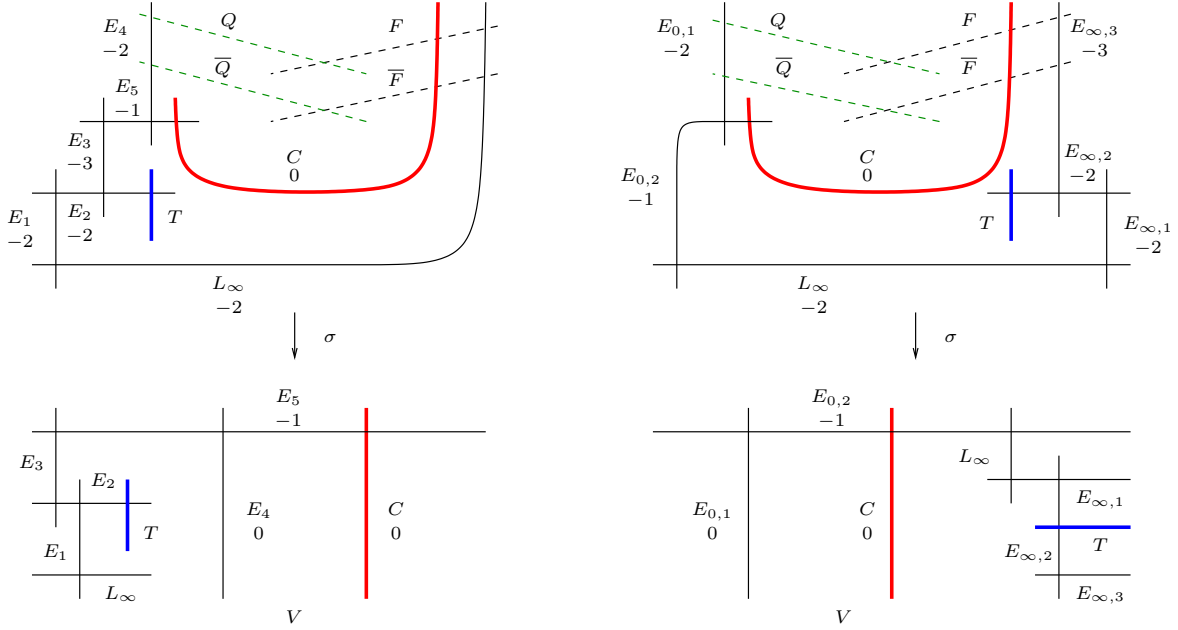} 
\caption{One singular point: case $A_6$ on the left and $A_5^*$ on the right.}  
\label{fig:onesing}
\end{figure}

4) If $p_{\infty}$ is an oscnode whose inverse image in the normalization of $C$ consists of a pair $\{q,\overline{q}\}$ of non-real
complex conjugate points and whose tangent is different from $L_\infty$. It follows from Lemma \ref{lem:RatEmbed-from-projcurve}
that $C\cap(L_{\infty}\setminus\{p_{\infty}\})$ consists of a unique
other smooth real point $p_{0}$ at which $C$ and $L_{\infty}$ intersect
with multiplicity $2$. As in the previous case, the tangent $T$
to $C$ at $p_{\infty}$ intersects $C$ with multiplicity $4$ at
$p_{\infty}$, and we consider the rational map $\overline{\rho}:\mathbb{P}_{\mathbb{R}}^{2}\dashrightarrow\mathbb{P}_{\mathbb{R}}^{1}$
defined by the pencil $\mathcal{M}\subset|\mathcal{O}_{\mathbb{P}_{\mathbb{R}}^{2}}(4)|$
generated by $C$ and $3T+L_{\infty}$. A minimal resolution $\tau:V_{1}\rightarrow\mathbb{P}_{\mathbb{R}}^{2}$
of the indeterminacy of $\overline{\rho}$ for which $B$ has no imaginary loop is obtained by blowing-up
on the one hand $p_{0}$ with exceptional divisors $E_{0,1}$ and
then the intersection point $p_{0,1}$ of the proper transform of $C$ and
$E_{0,1}$ with exceptional divisor $E_{0,2}$, and on the other hand
blowing-up $p_{\infty}$ with exceptional divisor $E_{\infty,1}$, then two
times successively the intersection point $p_{\infty,i}$ of the proper transform 
of $C$ and $E_{\infty,i}$ with exceptional divisor $E_{\infty,i+1}$
and finally blowing-up the pair $\{q,\overline{q}\}$ of non-real
complex conjugate points of intersection of $E_{\infty,3}$ with the
proper transform of $C$, with respective exceptional divisors $F$
and $\overline{F}$ (See right hand side of Figure \ref{fig:onesing}). The reduced total transform
$B_{1}$ of $L_{\infty}$ is a tree of smooth geometrically rational
curves. The proper transform of $C$ in $V$ is a smooth rational
curve with self-intersection $0$. So $\overline{\pi}_{1}=\overline{\rho}\circ\tau_{1}:V_{1}\rightarrow\mathbb{P}_{\mathbb{R}}^{1}$
is a $\mathbb{P}^{1}$-fibration having the proper transform of $C$
as a smooth real fiber and the curves $E_{0,2}$, $F$ and $\overline{F}$
as disjoint sections. The fiber of $\overline{\pi}_{1}$ over the
real point $\overline{\pi}_{1}(E_{0,1})$ consists of the union of
$E_{0,1}$ and the proper transforms of the non-real complex conjugate
smooth conics $Q$ and $\overline{Q}$ passing respectively through
$p_{0},p_{\infty},p_{\infty,1},p_{\infty,2},q$ and $p_{0},p_{\infty},p_{\infty,1},p_{\infty,2},\overline{q}$.
These are disjoint $(-1)$-curves in $V_{1,\mathbb{C}}$ which can
thus be contracted by a birational morphism $\sigma:V_{1}\rightarrow V$
defined over $\mathbb{R}$. We let $B=L_{\infty}\cup E_{\infty,1}\cup E_{\infty,2}\cup E_{\infty,3}\cup E_{0,1}\cup E_{0,2}$
be the proper transform of $B_{1,\mathbb{R}}$. The proper transform
of $E_{0,1}$ is then a full fiber of the induced $\mathbb{P}^{1}$-fibration
$\overline{\pi}=\overline{\pi}_{1}\circ\sigma^{-1}:V\rightarrow\mathbb{P}_{\mathbb{R}}^{1}$,
and the latter restricts on $S=V\setminus B$ to an $\mathbb{A}^{1}$-fibration
$\pi:S\rightarrow\mathbb{A}_{\mathbb{R}}^{1}$ with a unique singular
fiber equal to $T\cap S\simeq\mathbb{A}_{\mathbb{R}}^{1}$
with multiplicity $3$. We conclude as in the previous case that $f:\mathbb{A}{}_{\mathbb{R}}^{1}\dashrightarrow\mathbb{A}_{\mathbb{R}}^{2}$
is biddable but that none of the sufficient conditions for rectifiablity given Proposition \ref{thm:Main-rectif-criterion} is satisfied.
\end{proof}

\begin{rem}\label{rem:TwoQuarticModels} Let $f_i:\mathbb{A}_{\mathbb{R}}^{1}\dashrightarrow\mathbb{A}_{\mathbb{R}}^{2}$, $i=1,2$,
be  rational smooth embeddings whose associated curve $C\subset\mathbb{P}_{\mathbb{R}}^{2}$ is a quartic whose unique
singular point $p_{\infty}$  is an oscnode $A_{5}^{*}$ or a ramphoid cusp $A_{6}$ such that $T_{p_{\infty}}C\neq L_{\infty}$ respectively. 
In each case, the $\mathbb{A}^1$-fibered model $\pi_i:S_i\rightarrow \mathbb{A}^1_{\mathbb{R}}$ of $\mathbb{R}^2$ for which $\alpha_i \circ f_i: \mathbb{A}^1_{\mathbb{R}}\rightarrow S_i$ is a closed embedding as a smooth fiber of $\pi_i$ constructed in the proof of Lemma \ref{lem:unique-sing} has a unique multiple fiber, of multiplicity $3$, say $\pi_i^{-1}(0)$. On the other hand, one sees on Figure \ref{fig:onesing} that the $\mathbb{P}^1$-fibrations $\overline{\pi}_i:V_i\rightarrow \mathbb{P}^1_{\mathbb{R}}$ extending $\pi_i$ have non-isomorphic real fibers  $\overline{\pi}_1^{-1}(0)$ and $\overline{\pi}_2^{-1}(0)$: in the $A_{5}^{*}$ case, the irreducible component of $\overline{\pi}_1^{-1}(0)$ intersecting the section of $\overline{\pi}_1$ contained in $B_1$ has self-intersection $-3$, while the ireducible component of $\overline{\pi}_2^{-1}(0)$ intersecting the section of $\overline{\pi}_2$ contained in $B_2$ has self-intersection $-2$ in the $A_{6}$ case. This implies that there cannot exist any birational map $\beta:V_1\dashrightarrow V_2$ which restricts to a birational diffeomorphism between $S_1$ and $S_2$ and such that $\overline{\pi}_2\circ \beta=\overline{\pi}_1$. 
\end{rem}

\bibliographystyle{amsplain} 

\begin{thebibliography}{99}

\bibitem{AM75} S. S. Abhyankar and T. T. Moh, \emph{Embeddings of the line in the plane}, J. Reine Angew. Math. 276 (1975), 148-166.

\bibitem{Ar76} V. I. Arnold, \emph{Local  Normal  Forms  of  Functions}, Inventiones  Math.  35 (1976), 87-109.  

\bibitem{BR90}  R. Benedetti and J.-J. Risler, \emph{Real algebraic and semi-algebraic sets}, Actualit\'es Math\'ematiques. [Current   Mathematical Topics], Hermann, Paris, 1990.

\bibitem{BD17} J. Blanc and A. Dubouloz, \emph{Algebraic models of the real affine plane}, EPIGA Volume 2 (2018), Article Nr. 14. 


\bibitem{BiH07} I. Biswas and J. Huisman, \emph{Rational real algebraic models of topological surfaces}, Doc. Math. 12 (2007), 549-567.


\bibitem{Co59} J. L. Coolidge, \emph{A treatise on algebraic plane curves}, Dover, New York, 1959.



\bibitem{DuMa16} A. Dubouloz and F. Mangolte, \emph{Real frontiers of fake planes}, Eur. J. Math. 2(1),  (2016), 140-168.

\bibitem{DuMa17} A. Dubouloz and F. Mangolte, \emph{Fake Real Planes: exotic affine algebraic models of $\mathbb{R}^2$}, Selecta Math., 23(3), (2017), 1619-1668.


\bibitem{GuUtTaj66} D.A. Gudkov, G.A. Utkin and M.L. Taj,  \emph{The complete classification of irreducible curves of the 4 th order}, Mat.Sb., Nov.Ser., 69(111), N 2, (1966), 222-256. 

\bibitem{Gu88} D.A. Gudkov  \emph{Plane real projective quartic curves}, In: Viro O.Y., Vershik A.M. (eds) Topology and Geometry - Rohlin Seminar. Lecture Notes in Mathematics, vol 1346 (1988). Springer, Berlin, Heidelberg.

\bibitem{GMMR08} R.V. Gurjar,     K. Masuda,     M. Miyanishi,     P. Russell, \emph{Affine Lines on Affine Surfaces and the Makar--Limanov Invariant } Canad. J. Math. 60(2008), 109-139. 


\bibitem{Ii70} S. Iitaka, \emph{On {D}-dimensions of algebraic varieties}, Proc. Japan Acad. 46 (1970), 487-489.

\bibitem{Ii77} S. Iitaka,  \emph{On logarithmic Kodaira dimension of algebraic varieties}, Complex analysis and algebraic geometry, 175-189. Iwanami Shoten, Tokyo, 1977.

\bibitem{Ii83} S. Iitaka, \emph{On a characterization of two lines on a projective plane}, Proc. Algebraic Geometry, Lecture Notes in Math., lO16 (1983), 432-448, Springer.

\bibitem{Ii88} S. Iitaka,  \emph{Classification of reducible plane curves}, Tokyo J. Math. 11 (1988), no. 2, 363-379.

\bibitem{KaMi78} T. Kambayashi and M. Miyanishi, \emph{On flat fibrations by the affine line},    Illinois J. Math. Volume 22, Issue 4 (1978), 662-671.


\bibitem{KuMu83} N. M. Kumar and M. P. Murthy, \emph{Curves with negative self-intersection on rational surfaces}, J. Math. Kyoto Univ. 22 (1982/1983), 767-777.

\bibitem{Ma17} F. Mangolte, \emph{Variétés algébriques réelles}, Cours Sp\'ecialis\'es,
  vol.~24, Soci\'et\'e Math\'ematique de France, Paris, 2017. 

\bibitem{Ma33} J.I. Mayer, \emph{Projective Description of Plane Quartic Curves},  Tohoku Mathematical Journal, First Series,  36 (1933),1-21.

\bibitem{MiyBook} M. Miyanishi, \emph{Open {A}lgebraic {S}urfaces}, CRM Monogr. Ser., 12, Amer. Math. Soc., Providence, RI, 2001. 


\bibitem{Na62} M. Nagata, \emph{Imbedding of an abstract variety in a complete variety}, J. Math. Kyoto 2,  (1962), 1-10.

\bibitem{Na84} M. Namba, \emph{Geometry of Projective Algebraic Curves}, Marcel Dekker, Inc., New York, 1984.


\bibitem{vdE00} A. van den Essen,  \emph{Polynomial automorphisms and the Jacobian conjecture}, Progress in Mathematics, 190. Birkh\"a user Verlag, Basel, 2000.

\bibitem{Wa35} R.J. Walker, \emph{Reduction of the singularities of an algebraic surface}, Ann. Math. (2) 36(2), (1935), 336-365 .

\end{thebibliography}

\end{document}